\documentclass[12pt]{article}
\usepackage{graphicx}
\usepackage{amsthm, amsmath, amssymb}
\usepackage{fullpage}
\usepackage{enumitem}
\usepackage{lineno}
\usepackage[b]{esvect}
\usepackage{tikz}
\usepackage{float}
\usetikzlibrary{shapes.multipart,shapes.geometric,topaths,calc}
\usepackage{hyperref}
\newtheorem{theorem}{Theorem}
\newtheorem*{thm}{Theorem}
\newtheorem{lem}[theorem]{Lemma}

\newtheorem{corollary}{Corollary}
\newtheorem{prop}[theorem]{Proposition}
\newtheorem*{observation*}{Observation}

\theoremstyle{remark}

\theoremstyle{definition}

\newtheorem*{conjecture}{Conjecture}

\newcommand{\F}{{\cal F}}
\numberwithin{equation}{section}
\newcommand{\anc}{\boxplus}
\newcommand{\sq}{~\square~}

\begin{document}

 \title{Long path and cycle decompositions of even hypercubes}
\author{
 Maria Axenovich\thanks{Karlsruhe Institute of Technology, Karlsruhe, Germany, \texttt{maria.aksenovich@kit.edu}.} 
 \and David Offner\thanks{Carnegie Mellon University, Pittsburgh, PA, USA, \texttt{doffner@andrew.cmu.edu}.}
 \and Casey Tompkins\thanks{Discrete Mathematics Group, Institute for Basic Science (IBS), Daejeon, Republic of Korea, \texttt{ctompkins496@gmail.com}.}
 }
\date{\today}

\maketitle

\begin{abstract}
We consider edge decompositions of the $n$-dimensional hypercube $Q_n$ into isomorphic copies of a given graph $H$.  While a number of results are known about decomposing $Q_n$ into graphs from various classes, the simplest cases of paths and cycles of a given length are far from being understood.  A conjecture of Erde asserts that if $n$ is even, $\ell < 2^n$ and $\ell$ divides the number of edges of $Q_n$,  then the path of length $\ell$ decomposes $Q_n$.
Tapadia et al.\ proved that any path of length $2^mn$,  where $2^m<n$, satisfying these conditions decomposes $Q_n$.
Here, we make progress toward resolving Erde's conjecture by showing that cycles of certain lengths up to $2^{n+1}/n$ decompose $Q_n$.  As a consequence, we show that $Q_n$ can be decomposed into copies of any path of length at most $2^{n}/n$  dividing  the number of edges of $Q_n$, thereby settling Erde's conjecture up to a linear factor.  
\end{abstract}

\section{Introduction}
For any graph $G$, we denote its vertex set by $V(G)$ and its edge set by $E(G)$.
The \textit{$n$-dimensional hypercube} $Q_n$ is the graph with $V(Q_n) = \{0,1\}^n$ and edges between pairs of vertices that differ in exactly one coordinate.
Given graphs $G$ and $H$, we say that $H$ \textit{decomposes} $G$ if $G$ is a pairwise edge-disjoint union of isomorphic copies of $H$. For any  fixed graph $H$ which is a subgraph of some hypercube, Offner \cite{offner1} showed that $H$ almost decomposes any $Q_n$ for sufficiently large $n$. More precisely, a subgraph of $Q_n$ with all but at most $o(|E(Q_n)|)$ edges of $Q_n$ 
is a pairwise edge-disjoint union of isomorphic copies of $H$.  It was shown implicitly by Aubert and Schneider~\cite{aubert}  that when $n$ is even $Q_n$ has a decomposition into Hamiltonian cycles, see also Aspach et al. \cite{alspach} for an explicit statement. While some other research on graph decompositions allows paths and cycles of different lengths, for example \cite{GKO16}, we focus on decompositions of hypercubes into cycles and paths of given length.  

If $n$ is odd then each vertex of $Q_n$ has odd degree and hence must be an endpoint of some path in a path decomposition. This implies that there are at least $2^{n-1}$ paths in such a decomposition and the length of each such path is at most $|E(Q_n)|/2^{n-1} = n2^{n-1}/2^{n-1} = n$.  In fact, Anick and Ramras~\cite{AR} as well as Erde \cite{erde} proved that for odd $n$, $Q_n$ can be decomposed into copies of any path of length at most $n$ and dividing the number of edges in $Q_n$.  While for odd $n$, we can only hope for path decompositions into short paths, and the problem has been fully resolved, when $n$ is even, Erde   formulated the following strong conjecture that implies that there are path decompositions of hypercubes into long paths.

\begin{conjecture}[Erde \cite{erde}] 
 If $n$ is even, $\ell < 2^n$, and   $\ell$ divides the number of edges of $Q_n$,  then the path of length $\ell$ decomposes $Q_n$. 
\end{conjecture}

For even $n$, we know that $Q_n$ is decomposable into Hamiltonian cycles, so by dividing each cycle into paths of equal length, we see that $Q_n$ is decomposable into paths of length $2^{i}$ for $0 \le i \le n-1$. Tapadia et al.\ \cite{TWB} proved that for even $n$ and $m$ such that $2^m<n$, $Q_n$ is decomposable into paths of length at most $2^mn$. Erde noticed that if $n$ is even and $y$ is an odd divisor of $n$, then $Q_n$ is decomposable into paths of length $y2^{{n/y} -1}$.

Here, we prove that there are  cycle decompositions of hypercubes of even dimension into long cycles, from which it follows that there are decompositions of such hypercubes into long paths. The best known result on cycle decompositions is by Tapadia et al.\ \cite{TWB}  (see also  Horak et al. \cite{HSW}) which gives cycle decompositions of $Q_n$ into cycles of length at most $n^2$. 

\begin{thm}[Tapadia et al.\ \cite{TWB}]
Let $n$ and $m$ be integers where  $n$ is positive and even and $m$ is nonnegative,  such that 
 $ 2^{m}\leq n$. Then a cycle of length $2^{m}n$ decomposes $Q_n$.
\end{thm}

Note that the number of edges in $Q_n$ is $n2^{n-1}$. So  for even $n$, if there is a cycle decomposition of $Q_n$ into cycles of length $\ell$, then $\ell$ must divide $n2^{n-1}$, i.e. $\ell = y 2^{m}$, where $y$ is an odd divisor of $n$. We will show that for any odd divisor $y$ of $n$, there is a cycle decomposition of $Q_n$ into cycles of length $y 2^{m}$, where $m$ can take a range of values.

\begin{theorem}\label{main}
Let $n =xy2^\alpha$,  where $x, y \ge 1$ are odd, and $\alpha \ge 1$. 
Suppose $y$ has binary representation $y = 2^{i_1} + 2^{i_2} + \cdots + 2^{i_j}$, where $i_1 > i_2 > \cdots > i_j=0$. Then for any $q$,  $0 \le q \le 
n - i_1 - 2xj$, $Q_n$ has an edge decomposition into $x2^{i_1 + \alpha +j-2+q}$ cycles of length $y 2^{n-i_1-j-q+1}$.
\end{theorem}

As an example, consider $Q_{30}$, where $\alpha = 1$. Letting $x=3$ and $y=5 = 2^2 + 2^0$ gives $i_1 = 2$ and $j = 2$, so $n- i_1 - 2xj = 16$.  Thus we get decompositions of $Q_{30}$ into $x2^{i_1 + \alpha + j - 2 +q} = 3 \cdot 2^{3+q}$
cycles, for $0 \le q \le 16$.  Since $Q_{30}$ has $30\cdot 2^{29}$ edges, the cycle lengths of these decompositions are 
$\{30\cdot 2^{29}/3\cdot 2^n ~ : ~ 3 \le n \le 19\} = \{5\cdot 2^{m} ~ : ~ 11 \le m \le 27\}$. In particular, when $q=0$, we obtain a decomposition of $Q_{30}$ into only  $24$ cycles of length $5\cdot 2^{27}$, whereas the smallest possible number of cycles in a decomposition  of $Q_{30}$ is $15$. See Table~\ref{tablemain} in the appendix for further numerical examples.

 By taking $q=0$  in Theorem~\ref{main}, we prove the following Corollary in Section~\ref{importantproofs}.

\begin{corollary}\label{upshot}
Let $n$ be even and $y$ be an odd divisor of $n$.  Then there is a decomposition of $Q_n$ into cycles of length $\ell$, where $\ell \geq 2^{n+1}/n$ and $\ell=y2^m$ for some $m$. 
\end{corollary}

Note that dividing each cycle in a cycle  decomposition of $Q_n$ into paths of equal length creates a path decomposition of $Q_n$, so in Section~\ref{importantproofs} we prove the following theorem for path decompositions.

\begin{theorem}\label{paths}
Let $n$ be even and  $\ell$ be a positive integer such that  $\ell\leq 2^n/n$ and $\ell$ divides the number of edges in $Q_n$. Then there is a decomposition of $Q_n$ into paths of length $\ell$.
\end{theorem}

The rough idea of the proof of Theorem~\ref{main} is as follows.  We represent  $Q_n$  as a Cartesian product of smaller hypercubes. By induction, using the decomposition of the hypercube into Hamiltonian cycles as a base case, we decompose each of the smaller hypercubes into cycles.  We consider the products of these cycles from different copies of the smaller hypercubes. The Cartesian product of two cycles forms a toroidal grid (which we refer to simply as a torus), and in Section~\ref{sectionkwig}  we show how to decompose a torus into several cycles of the same length using what we call a ``wiggle'' decomposition.   Most hypercubes cannot be decomposed into tori, but using the notion of ``representing sets'' of the vertices in a cycle, we show how to decompose the hypercube into graphs which are tori that are subdivided in a nicely structured way.  We then show that our wiggle decomposition can also decompose these subdivided tori into cycles of all the same length.  While previous researchers have considered decompositions of tori into cycles, the use of the subdivided tori and representing sets in this paper are the key to decomposing the hypercube into long cycles.

Theorems~\ref{main} and \ref{paths} show that the hypercube can be decomposed into long cycles and paths, respectively, but our understanding is still incomplete for the longest cycles and paths.  It is still open which cycles of lengths between $2^{n+1}/n$ and $2^n$ can decompose $Q_n$.  For example, by our methods we cannot construct a decomposition of $Q_{14}$ into cycles of length $7 \cdot 2^{11}$.  Even if we could decompose the hypercube into all possible cycles, this would not completely resolve Erde's conjecture, as paths in $Q_n$ of length greater than $2^{n-1}$ cannot be obtained by evenly dividing cycles in a cycle decomposition.  For example, Erde conjectures that $Q_6$ should have a decomposition into paths of length 48, but since $Q_6$ has only 64 vertices, there are no cycles long enough to be divided into paths of length 48.

The rest of the paper is structured as follows. We give more background and historical information on decompositions of hypercubes in Section \ref{background}. In Section~\ref{splittable}, we introduce the wiggle decomposition for decomposing tori and subdivided tori into cycles. We also introduce stronger notions of splittable and DR-splittable decompositions, and show how to produce these types of cycle decompositions of tori and subdivided tori.
In Section \ref{combine} we state several general decomposition results on Cartesian products, and show how given cycle decompositions of graphs $G$ and $G'$ we can produce a cycle decomposition of their Cartesian product with all cycles the same length. We prove Theorem~\ref{main} and Theorem~\ref{paths}, as well as Corollary \ref{upshot}  in Section~\ref{importantproofs}.
We conclude with a refinement of the main result in Section \ref{refinement}.


\section{Background}\label{background}


For a graph $G=(V, E)$, we say that a graph $H$ {\it divides} the graph $G$ if the greatest common divisor of the degrees of $H$ divides the greatest common divisor of the degrees of $G$ and $|E(H)|$ divides $|E(G)|$.
We call a subgraph of $G$ isomorphic to $H$ a {\it copy } of $H$ in $G$. 
We use $K_n$ to denote a complete graph on $n$ vertices.  A classical theorem of Wilson~\cite{wilson} states that for any graph $H$, if $n$ is sufficiently large and $H$ divides $K_n$ then $H$ decomposes $K_n$.  This result was generalized for subgraphs $G$  of $K_n$ with sufficiently large minimum degree and graphs $H$ dividing $G$, see Keevash \cite{k}, Glock et al.\ \cite{g}, and Kim et al.\ \cite{KKOT19}.  Given Wilson's result on $K_n$, it is natural to consider the analogous problem with other ground graphs, for example a hypercube. 

A graph $H$ is called \textit{cubical} if it is a subgraph of $Q_n$ for some $n$.  
It is clear that only graphs which are  cubical and divide $Q_n$  can decompose $Q_n$.
However, unlike the above results for dense subgraphs of $K_n$, these properties are not sufficient for decomposing $Q_n$, as shown by a counterexample of Bonamy et al.\ \cite{BMS}.

The initial results involving packings and decompositions of the hypercube are due to Stout~\cite{stout} and were motivated by  processor allocation problems.  He introduced both the notion of a vertex packing and an edge packing of the hypercube and proved an asymptotically optimal result for vertex packing.  He showed that for any cubical graph $H$, there are pairwise vertex disjoint copies of $H$ in $Q_n$ containing all but $o(|V(Q_n)|)$ vertices of $Q_n$. 
Answering a question of Offner, Gruslys~\cite{gruslys} strengthened Stout's result on vertex packing by proving that if the order of $H$  is a power of $2$, then for sufficiently large $n$,  there are pairwise vertex-disjoint copies of $H$ containing all vertices of $Q_n$. In fact, Gruslys's result holds even for the stronger notion of isometric embeddings.

 Stout~\cite{stout} proved a number of results about edge packing of graphs in $Q_n$. For example, he showed that if $T$ is a tree with $n$ edges, then $T$ decomposes $Q_n$, a result independently proved by Fink~\cite{fink}.  Stout conjectured that for any cubical graph $H$ there are pairwise edge-disjoint copies of $H$ in $Q_n$ containing all but $o(|E(Q_n)|)$ edges of $Q_n$. This conjecture was later proved by Offner~\cite{offner1}.  A fan  with a root vertex  $v$ is a graph which is a union of  cycles of the same length that pairwise  share only $v$. A double-fan  is the graph obtained by joining the root vertices of two vertex disjoint fans by an edge.  In~\cite{roy}, Roy and Kureethara proved several results about decomposing $Q_n$ into fans and double-fans.
Horak et al.\   \cite{HSW} showed that if  $H$ is a cubical graph of size $n$, each block of which is either a cycle or an edge, then  $H$ decomposes $Q_n$. 

A major direction in the decomposition literature concerns Hamiltonian decompositions,  that is decompositions  into Hamiltonian cycles or Hamiltonian cycles and a perfect matching, see for example a survey of Alspach et al.\ \cite{alspach}.  Investigations of Hamiltonian decompositions of $K_n$ were carried out as early as the 1800's by Walecki in~\cite{lucas}.   His constructions showed that $K_n$ has a Hamiltonian decomposition for all $n$ and a decomposition into Hamiltonian paths for even $n$. This result was extended by Auerbach and Laskar~\cite{auerbach}, who showed that complete multipartite graphs with parts of equal size have Hamiltonian decompositions.   
Ringel \cite{ringel} proved that $Q_n$ has a Hamiltonian decomposition for all integers $n$ which are powers of $2$ and asked whether $Q_n$ has a Hamiltonian decomposition for all even $n$.

Closely relevant to cycle decompositions of $Q_n$ are  Hamiltonian cycle decompositions of the product of cycles. Kotzig \cite{kotzig} proved that the Cartesian product of any two cycles is decomposable into Hamiltonian cycles. This result was extended to products of three cycles by Foregger \cite{foregger}, who in the process gave an alternative proof of Kotzig's result.  Finally,  Aubert and Schneider~\cite{aubert} extended Foregger's result by proving a general theorem which implies that a product of arbitrarily many cycles has a Hamiltonian decomposition.  One consequence of their results is a solution to Ringel's problem of showing that $Q_n$ has a Hamiltonian decomposition when $n$ is even, since $Q_n$ is the Cartesian product of $n/2$ cycles of length 4.
  
An important  open problem for Hamiltonian decompositions is a conjecture of Bermond~\cite{bermond} asserting that the Cartesian product of two graphs, each having a  Hamiltonian decomposition, has a Hamiltonian decomposition. This conjecture has been settled under fairly general conditions by Stong~\cite{strong} but remains open in general. Motivated by problems in parallel computing, Bass and Sudborough \cite{bass} considered decompositions of $Q_n$ into $k$-regular spanning subgraphs.


\section{Definitions and notation}
\subsection{Basic definitions}

For graphs $G$ and $H$, denote by $G \cup H$ the graph with $V(G \cup H) = V(G) \cup V(H)$ and $E(G \cup H) = E(G) \cup E(H)$.
We denote by $G\sq H$ the Cartesian product of $G$ and $H$, i.e., a graph with vertex set $\{(u,v): u\in V(G), v\in V(H)\}$
and edge set $\{(u,v)(u',v'):  u=u', vv'\in E(H) \text{ or } v=v', uu'\in E(G)\}$. We use the notation $(e, v)$ and $(u, e')$ for 
an edge $(u,v)(u',v)$, $e=uu'$ and for an edge $(u,v)(u,v')$, $e'=vv'$, respectively.
In our drawings of $G \sq H$, we represent $V(G\sq H)= V(G) \times V(H)$ as a rectangular grid with copies of $V(G)$ forming columns or subsets of vertical lines and copies of $V(H)$ forming rows or subsets of horizontal lines. 
Then the edges of $G$ are represented vertically, and $H$ horizontally, so we call an edge of the form $(e,v)$ a ``vertical" edge and one of the form $(u, e')$ a ``horizontal" edge.
For a fixed $e\in E(G)$, we call the set of edges $\{(e, v): v\in V(H)\}$ an {\it edge row} or just a {\it row}. 
Similarly, for a fixed $e'\in E(H)$, we call the edges $\{(u, e'): u\in V(G)\}$ an {\it edge column} or just a {\it column}. Note that in our convention the edges in a row are oriented vertically, and those in a column are oriented horizontally. If $G_1, \ldots ,G_k$ are subgraphs of $G$, we say the set of graphs $\{G_1, \ldots ,G_k\}$ forms a \textit{decomposition} of $G$ if $G=G_1 \cup \cdots \cup G_k$ and the subgraphs are pairwise edge-disjoint. We say the decomposition is a \textit{cycle decomposition} if $G_1, \ldots ,G_k$ are all cycles.  In this paper we are interested in cycle decompositions where all of the cycles have the same length.

\subsection{Anchored products of graphs and subdivided tori}\label{anchored}
Given graphs $G$ and $G'$ and $S \subseteq V(G)$, $S' \subseteq V(G')$,  we define the \textit{anchored product} $(G,S) \anc (G',S')$ of the pairs $(G,S)$ and $(G',S')$ to be the graph with the vertex set 
$$\{(u,v) : u \in V(G), v \in V(G'), \text{ and } u \in S \text{ or } v \in S'\}$$
and edge set
$$  \{(u,v)(u',v') : uu' \in E(G), ~v=v' \in S'\} \cup  \{(u,v)(u',v') : u=u' \in S, ~vv' \in E(G')\},$$
\noindent
see Figure~\ref{anchorpic}. Note that if $S=V(G)$ and $S' = V(G')$, the anchored product $(G,S) \anc (G',S')$ is the same graph as the Cartesian product $G \sq G'$.
Alternatively, we see that 
$$(G,S) \anc (G',S') = G\sq G'[ (S\times V(G')) \cup (V(G) \times S)],$$
where for a graph $F$ and a vertex subset $X \subseteq V(F)$, the notation $F[X]$ stands for the subgraph of $F$ induced by $X$.

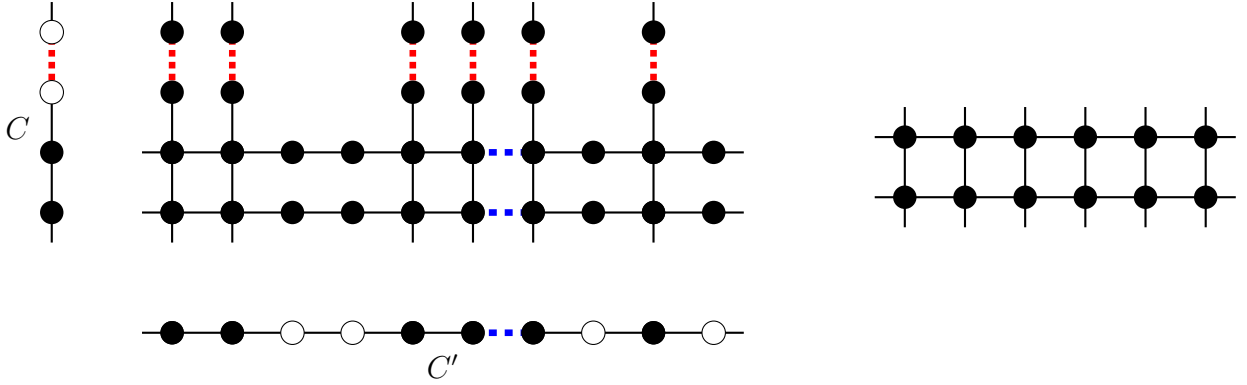
\begin{figure}
\begin{center}

\begin{tikzpicture}[scale=.8]  

    \foreach \x in {-2, 0, 1, 4, 5, 6, 8}{
    \draw[thick] (\x,-.5) -- (\x,2);
    \draw[thick] (\x,3) -- (\x,3.5);
}

    \foreach \y in {-2, 0, 1}{
    \draw[thick] (-.5,\y) -- (5,\y);
        \draw[thick] (6,\y) -- (9.5,\y);
}

    \foreach \x in {-2, 0, 1, 4, 5, 6, 8}{
    \draw[red, dotted, line width=2.5] (\x,2) -- (\x,3);
}

    \foreach \y in {-2, 0, 1}{
    \draw[blue, dashed, line width=2.5] (5,\y) -- (6,\y);
}

    \foreach \x in {0,1, 4, 5, 6, 8}{
      \node[style={circle,fill=black},scale=.8] at (\x,-2){};
}

    \foreach \x in {0,1, 4, 5, 6, 8}{
    \foreach \y in {0, ..., 3}{
      \node[style={circle,fill=black},scale=.8] at (\x,\y){};
}
}

    \foreach \x in {0, ..., 9}{
    \foreach \y in {0,1}{
      \node[style={circle,fill=black},scale=.8] at (\x,\y){};
}
}

    \foreach \x in {0,1, 4, 5, 6, 8}{
      \node[style={circle,fill=black},scale=.8] at (\x,-2){};
}

    \foreach \x in {2, 3, 7, 9}{
      \node[style={circle, draw, fill=white},scale=.8] at (\x,-2){};
}

    \foreach \y in {0,1}{
      \node[style={circle,fill=black},scale=.8] at (-2,\y){};
}

    \foreach \y in {2,3}{
      \node[style={circle, draw, fill=white},scale=.8] at (-2,\y){};
}

      \node (0) at (-2,1.4) [label=left:$C$]{};
      \node (1) at (4.5,-2) [label=below:$C'$]{};

\end{tikzpicture}
\hfill
\begin{tikzpicture}[scale=.8]  

    \foreach \x in {0,...,5}{
    \draw[thick] (\x,-.5) -- (\x,1.5);
}
    \foreach \y in {0,1}{
    \draw[thick] (-.5,\y) -- (5.5,\y);
}

    \foreach \x in {0, ..., 5}{
    \foreach \y in {0,1}{
      \node[style={circle,fill=black},scale=.8] at (\x,\y){};
}
}

      \node (0) at (0,-3) {};

\end{tikzpicture}
\end{center}
\caption{Left: Two cycles $C=(0, 1,2,3,0)$ and $C'= (0,1,2,3, 4, 5, 6, 7, 8, 9, 0)$, and the anchored product  $(C, S)\anc (C', S')$, where  $S=\{0,1\}$ and $S'= \{0, 1, 4, 5,6, 8\}$. A given row and column of the anchored product are highlighted in dotted red and dashed blue, respectively, along with the corresponding edge from the original cycle. Right: The underlying torus of $(C, S)\anc (C', S')$.} \label{anchorpic}
\end{figure}

We call the Cartesian product of two cycles $C \sq C'$ a \textit{torus}. Given $v \in V(C)$, we call the cycle induced in $C \sq C'$ by $\{v\} \times V(C')$ a \textit{horizontal cycle}, and given  $v' \in V(C')$, we call the cycle induced in $C \sq C'$ by $V(C) \times \{v'\}$ a \textit{vertical cycle}. A \textit{subdivided torus} is a graph obtained from a torus by subdividing edges so that all edges in each row are subdivided by the same number of vertices and all edges in each column are subdivided by the same number of vertices. The number of subdivisions may be different in different rows or columns.  More formally, a graph $F$ is a subdivided torus if for some cycles 
$C$ and $C'$ and vertex sets $S \subseteq V(C)$ and $S' \subseteq V(C')$, $F=(C,S) \anc (C',S')$.  Note that a vertex has degree four in a subdivided torus if and only if it is in $S \times S'$, and otherwise it has degree two.
We also see that a subdivided torus is a subgraph of a larger torus $C\sq C'$ and a subdivision of a smaller torus obtained by contracting all degree two vertices. We refer to this smaller torus as the
{\it underlying torus} of the subdivided torus. Note that the underlying torus of $F$ is a Cartesian product of two cycles with lengths $|S|$ and $|S'|$, respectively. The set of edges of a row of $C\sq C'$ that are in $F$ is called a {\it row of a subdivided torus}. The columns are defined similarly.  Figures~\ref{anchorpic} and \ref{2wiggle} show examples of subdivided tori along with their underlying tori. Note that, as in Figure~\ref{anchorpic}, the underlying torus may be a product of a cycle of length 2 with another cycle.



\section{Cycle decompositions of tori and subdivided tori}\label{splittable}


\subsection{The $k$-wiggle decomposition on tori and subdivided tori}\label{sectionkwig}
Let $k \ge 2$ be an integer. We define a method for decomposing a torus that is product of a cycle $C$ of length divisible by $k$ and a cycle $C'$ of length at least $k$ and  congruent to $k \pmod 2$ into $k$ cycles of equal length called the \textit{$k$-wiggle decomposition}.   Let $C = (0,1, \ldots,  n-1, 0)$ be a cycle of length $n$ and $C' = (0,1, \ldots, m-1,0)$ a cycle of length $m$, where $k$ is a divisor of $n$ and $m=2s + k$ for some integer $s \ge 0$.  We say that a torus $T$  {\it allows the $k$-wiggle decomposition} if it meets these conditions.
In the important case $k=2$,  the condition for allowing the $k$-wiggle decomposition is equivalent to $n$ and $m$ being even. A decomposition of  the torus $C \sq C'$ into $k$ cycles $C_1, \ldots, C_{k}$, is called the $k$-\textit{wiggle decomposition}, if for $\ell =1, \ldots, k$,

\begin{align*}
E(C_\ell) = &\{ (i,j)({i+1},j) : ~ 0 \le j \le m-k-1,~ i \equiv \ell \pmod{k} \}\\
& \cup \{ ({i},j)({i+1},j) : ~0 \le p \le k-1,  ~i \equiv \ell + p \pmod{k}, ~j=m-k +p \} \\ 
&\cup \{ (i,j)(i,{j+1}) : ~i \equiv \ell \pmod{k} , ~0 \le j \le m-k-1, ~j \text{ odd} \} \\
&\cup \{ (i,j)(i,{j+1}) :~ i \equiv \ell +1 \pmod{k} , ~0 \le j \le m-k-1, ~j \text{ even} \} \\
&\cup \{ (i,j)(i,{j+1}) : ~0 \le p \le k-1, j=m-k +p, ~i \equiv \ell + p +1 \pmod{k} \}. \\
\end{align*}
See Figure~\ref{2and4wiggle16} for examples of the $k$-wiggle decomposition on Cartesian products of cycles for various $k$. The term wiggle comes from the fact that, when drawn, each of the $k$ cycles wiggle across the torus, before rising $k$ levels in a staircase pattern to repeat the wiggle $k$ levels above. Note that all cycles in a $k$-wiggle decomposition on a torus have the same length, and further, the cycles are all vertical translations of each other, i.e. the vertex $(i,j) \in V(C_1)$ if and only if the vertex $(i+\ell-1, j) \in V(C_\ell)$ and the edge $(i,j)(i',j') \in E(C_1)$ if and only if the edge $(i +\ell-1, j)(i'+\ell-1, j')  \in E(C_\ell)$. Finally, note that every vertex in the torus is in the vertex set of exactly two cycles $C_\ell$ and $C_{\ell+1}$, where we take the subscripts modulo~$k$.

\begin{figure}

\begin{tikzpicture}[scale=.4]  

    \foreach \x in {4, ..., 15}{
    \draw[thick, red] (\x,3.5) -- (\x,15.5);
}
    \foreach \y in {4, ..., 15}{
    \draw[thick, red] (3.5,\y) -- (15.5,\y);
}
%

    \draw[thick, ultra thick]  (3.5,4) -- (4,4) -- (4,5) -- (5,5) -- (5,4) -- (6,4) -- (6,5) -- (7,5) -- (7,4) -- (8,4) -- (8,5) -- (9,5) -- (9,4) -- (10,4) -- (10,5) -- (11,5) -- (11,4) -- (12,4) --(12,5) -- (13,5) -- (13,4) --(14,4) --(14,5) --(15,5) --(15,6) --(15.5,6) ;
    \draw[thick, ultra thick] (3.5,6) -- (4,6) -- (4,7) -- (5,7) -- (5,6) -- (6,6) -- (6,7) -- (7,7) -- (7,6) -- (8,6) -- (8,7) -- (9,7) -- (9,6) -- (10,6) -- (10,7) -- (11,7) -- (11,6) -- (12,6) --(12,7) -- (13,7) -- (13,6) --(14,6) --(14,7) --(15,7) --(15,8) --(15.5,8) ;

    \draw[thick, ultra thick]  (3.5,8) -- (4,8) -- (4,9) -- (5,9) -- (5,8) -- (6,8) -- (6,9) -- (7,9) -- (7,8) -- (8,8) -- (8,9) -- (9,9) -- (9,8) -- (10,8) -- (10,9) -- (11,9) -- (11,8) -- (12,8) --(12,9) -- (13,9) -- (13,8) --(14,8) --(14,9) --(15,9) --(15,10) --(15.5,10) ;
    \draw[thick, ultra thick] (3.5,10) -- (4,10) -- (4,11) -- (5,11) -- (5,10) -- (6,10) -- (6,11) -- (7,11) -- (7,10) -- (8,10) -- (8,11) -- (9,11) -- (9,10) -- (10,10) -- (10,11) -- (11,11) -- (11,10) -- (12,10) --(12,11) -- (13,11) -- (13,10) --(14,10) --(14,11) --(15,11) --(15,12) --(15.5,12) ;

    \draw[thick, ultra thick]  (3.5,12) -- (4,12) -- (4,13) -- (5,13) -- (5,12) -- (6,12) -- (6,13) -- (7,13) -- (7,12) -- (8,12) -- (8,13) -- (9,13) -- (9,12) -- (10,12) -- (10,13) -- (11,13) -- (11,12) -- (12,12) --(12,13) -- (13,13) -- (13,12) --(14,12) --(14,13) --(15,13) --(15,14) -- (15.5,14);
    \draw[thick, ultra thick]  (3.5,14) -- (4,14) -- (4,15) -- (5,15) -- (5,14) -- (6,14) -- (6,15) -- (7,15) -- (7,14) -- (8,14) -- (8,15) -- (9,15) -- (9,14) -- (10,14) -- (10,15) -- (11,15) -- (11,14) -- (12,14) --(12,15) -- (13,15) -- (13,14) -- (14,14) -- (14,15) -- (15,15) -- (15,15.5);

  \draw[thick, ultra thick] (15,3.5) -- (15,4) -- (15.5,4);

    \foreach \x in {4, ..., 15}{
    \foreach \y in {4, ..., 15}{
      \node[style={circle, draw, fill=white},scale=.4] at (\x,\y){};
}
}

\end{tikzpicture}
\hfill
\begin{tikzpicture}[scale=.4]  

    \foreach \x in {4, ..., 14}{
    \draw[thick, red] (\x,3.5) -- (\x,15.5);
}
    \foreach \y in {4, ..., 15}{
    \draw[thick, red] (3.5,\y) -- (14.5,\y);
}
%

    \draw[ultra thick]  (3.5,4) -- (4,4) -- (4,5) -- (5,5) -- (5,4) -- (6,4) -- (6,5) -- (7,5) -- (7,4) -- (8,4) -- (8,5) -- (9,5) -- (9,4) -- (10,4) -- (10,5) -- (11,5) -- (11,4) -- (12,4) --(12,5) -- (13,5) -- (13,6) --(14,6) --(14,7) --(14.5,7);
    \draw[ thick, blue]  (3.5,5) -- (4,5) -- (4,6) -- (5,6) -- (5,5) -- (6,5) -- (6,6) -- (7,6) -- (7,5) -- (8,5) -- (8,6) -- (9,6) -- (9,5) -- (10,5) -- (10,6) -- (11,6) -- (11,5) -- (12,5) --(12,6) -- (13,6) -- (13,7) --(14,7) --(14,8) --(14.5,8);

    \draw[ultra thick]  (3.5,7) -- (4,7) -- (4,8) -- (5,8) -- (5,7) -- (6,7) -- (6,8) -- (7,8) -- (7,7) -- (8,7) -- (8,8) -- (9,8) -- (9,7) -- (10,7) -- (10,8) -- (11,8) -- (11,7) -- (12,7) --(12,8) -- (13,8) -- (13,9) --(14,9) --(14,10) --(14.5,10);
    \draw[thick, blue]  (3.5,8) -- (4,8) -- (4,9) -- (5,9) -- (5,8) -- (6,8) -- (6,9) -- (7,9) -- (7,8) -- (8,8) -- (8,9) -- (9,9) -- (9,8) -- (10,8) -- (10,9) -- (11,9) -- (11,8) -- (12,8) --(12,9) -- (13,9) -- (13,10) --(14,10) --(14,11) --(14.5,11);

    \draw[ultra thick] (3.5,10) -- (4,10) -- (4,11) -- (5,11) -- (5,10) -- (6,10) -- (6,11) -- (7,11) -- (7,10) -- (8,10) -- (8,11) -- (9,11) -- (9,10) -- (10,10) -- (10,11) -- (11,11) -- (11,10) -- (12,10) --(12,11) -- (13,11) -- (13,12) --(14,12) --(14,13) --(14.5,13);
    \draw[thick, blue]  (3.5,11) -- (4,11) -- (4,12) -- (5,12) -- (5,11) -- (6,11) -- (6,12) -- (7,12) -- (7,11) -- (8,11) -- (8,12) -- (9,12) -- (9,11) -- (10,11) -- (10,12) -- (11,12) -- (11,11) -- (12,11) --(12,12) -- (13,12) -- (13,13) --(14,13) --(14,14) --(14.5,14);

    \draw[ultra thick]  (3.5,13) -- (4,13) -- (4,14) -- (5,14) -- (5,13) -- (6,13) -- (6,14) -- (7,14) -- (7,13) -- (8,13) -- (8,14) -- (9,14) -- (9,13) -- (10,13) -- (10,14) -- (11,14) -- (11,13) -- (12,13) --(12,14) -- (13,14) -- (13,15) --(14,15) --(14,15.5) ;
    \draw[thick, blue]  (3.5,14) -- (4,14) -- (4,15) -- (5,15) -- (5,14) -- (6,14) -- (6,15) -- (7,15) -- (7,14) -- (8,14) -- (8,15) -- (9,15) -- (9,14) -- (10,14) -- (10,15) -- (11,15) -- (11,14) -- (12,14) --(12,15) -- (13,15) -- (13,15.5);

  \draw[ultra thick] (14,3.5) -- (14,4) -- (14.5,4);
  \draw[thick, blue] (13,3.5) -- (13,4) -- (14,4) -- (14,5) -- (14.5,5);

    \foreach \x in {4, ..., 14}{
    \foreach \y in {4, ..., 15}{
      \node[style={circle, draw, fill=white},scale=.4] at (\x,\y){};
}
}

\end{tikzpicture}
\hfill
\begin{tikzpicture}[scale=.4]  

    \foreach \x in {4, ..., 15}{
    \draw[thick, purple] (\x,3.5) -- (\x,15.5);
}
    \foreach \y in {4, ..., 15}{
    \draw[thick, purple] (3.5,\y) -- (15.5,\y);
}
%

    \draw[ultra thick]  (3.5,4) -- (4,4) -- (4,5) -- (5,5) -- (5,4) -- (6,4) -- (6,5) -- (7,5) -- (7,4) -- (8,4) -- (8,5) -- (9,5) -- (9,4) -- (10,4) -- (10,5) -- (11,5) -- (11,4) -- (12,4) --(12,5) -- (13,5) -- (13,6) --(14,6) --(14,7) --(15,7) --(15,8) --(15.5,8) ;
    \draw[thick, red]  (3.5,5) -- (4,5) -- (4,6) -- (5,6) -- (5,5) -- (6,5) -- (6,6) -- (7,6) -- (7,5) -- (8,5) -- (8,6) -- (9,6) -- (9,5) -- (10,5) -- (10,6) -- (11,6) -- (11,5) -- (12,5) --(12,6) -- (13,6) -- (13,7) --(14,7) --(14,8) --(15,8) --(15,9) --(15.5,9) ;
    \draw[thick, blue] (3.5,6) -- (4,6) -- (4,7) -- (5,7) -- (5,6) -- (6,6) -- (6,7) -- (7,7) -- (7,6) -- (8,6) -- (8,7) -- (9,7) -- (9,6) -- (10,6) -- (10,7) -- (11,7) -- (11,6) -- (12,6) --(12,7) -- (13,7) -- (13,8) --(14,8) --(14,9) --(15,9) --(15,10) --(15.5,10) ;

    \draw[ultra thick]  (3.5,8) -- (4,8) -- (4,9) -- (5,9) -- (5,8) -- (6,8) -- (6,9) -- (7,9) -- (7,8) -- (8,8) -- (8,9) -- (9,9) -- (9,8) -- (10,8) -- (10,9) -- (11,9) -- (11,8) -- (12,8) --(12,9) -- (13,9) -- (13,10) --(14,10) --(14,11) --(15,11) --(15,12) --(15.5,12) ;
    \draw[thick, red]  (3.5,9) -- (4,9) -- (4,10) -- (5,10) -- (5,9) -- (6,9) -- (6,10) -- (7,10) -- (7,9) -- (8,9) -- (8,10) -- (9,10) -- (9,9) -- (10,9) -- (10,10) -- (11,10) -- (11,9) -- (12,9) --(12,10) -- (13,10) -- (13,11) --(14,11) --(14,12) --(15,12) --(15,13) --(15.5,13) ;
    \draw[thick, blue] (3.5,10) -- (4,10) -- (4,11) -- (5,11) -- (5,10) -- (6,10) -- (6,11) -- (7,11) -- (7,10) -- (8,10) -- (8,11) -- (9,11) -- (9,10) -- (10,10) -- (10,11) -- (11,11) -- (11,10) -- (12,10) --(12,11) -- (13,11) -- (13,12) --(14,12) --(14,13) --(15,13) --(15,14) --(15.5,14) ;

    \draw[ultra thick]  (3.5,12) -- (4,12) -- (4,13) -- (5,13) -- (5,12) -- (6,12) -- (6,13) -- (7,13) -- (7,12) -- (8,12) -- (8,13) -- (9,13) -- (9,12) -- (10,12) -- (10,13) -- (11,13) -- (11,12) -- (12,12) --(12,13) -- (13,13) -- (13,14) --(14,14) --(14,15) --(15,15) --(15,15.5);
    \draw[thick, red]  (3.5,13) -- (4,13) -- (4,14) -- (5,14) -- (5,13) -- (6,13) -- (6,14) -- (7,14) -- (7,13) -- (8,13) -- (8,14) -- (9,14) -- (9,13) -- (10,13) -- (10,14) -- (11,14) -- (11,13) -- (12,13) --(12,14) -- (13,14) -- (13,15) --(14,15) --(14,15.5) ;
    \draw[thick, blue]  (3.5,14) -- (4,14) -- (4,15) -- (5,15) -- (5,14) -- (6,14) -- (6,15) -- (7,15) -- (7,14) -- (8,14) -- (8,15) -- (9,15) -- (9,14) -- (10,14) -- (10,15) -- (11,15) -- (11,14) -- (12,14) --(12,15) -- (13,15) -- (13,15.5);

  \draw[ultra thick] (15,3.5) -- (15,4) -- (15.5,4);
  \draw[thick, red] (14,3.5) -- (14,4) -- (15,4) --(15,5) -- (15.5,5);
  \draw[thick, blue] (13,3.5) -- (13,4) -- (14,4) --(14,5) -- (15,5) -- (15,6) -- (15.5,6);

    \foreach \x in {4, ..., 15}{
    \foreach \y in {4, ..., 15}{
      \node[style={circle, draw, fill=white},scale=.4] at (\x,\y){};
}
}

\end{tikzpicture}

\caption{Left: The 2-wiggle decomposition of $C_{12} \sq C_{12}$ into two cycles (red and thick black).  Middle: The 3-wiggle decomposition of $C_{12} \sq C_{11}$ into three cycles (red, blue, and thick black). Right: The 4-wiggle decomposition of $C_{12} \sq C_{12}$ into four cycles (purple, blue, red, and thick black).}\label{2and4wiggle16}
\end{figure}
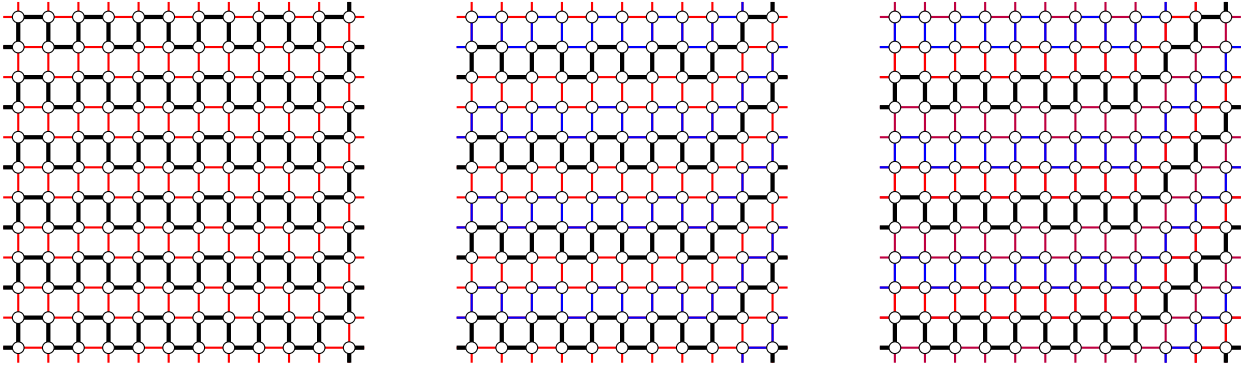

Consider now a subdivided torus $F=(C,S) \anc (C',S')$ such that its underlying torus $T$ allows a $k$-wiggle decomposition, i.e., $|S|$ is a multiple of $k$ and $|S'|$ is at least $k$ and congruent to $k$ modulo $2$. 
We define  a $k$-wiggle decomposition of $F$ as a decomposition obtained from the $k$-wiggle decomposition of $T$ by subdividing respective edges. More precisely, if an edge $e$ is in the $i$th cycle of the decomposition of $T$, we let all edges of $F$ obtained by subdividing $e$ be in the $i$th cycle of the decomposition of $F$. See Figure \ref{2wiggle}. We say a subdivided torus allows a $k$-wiggle decomposition if its underlying torus does.

The $k$-wiggle decomposition on a subdivided torus may not produce cycles of all the same length, for example if exactly one vertical edge of $C$ is subdivided. Next we give sufficient conditions on the subdivided torus to guarantee the cycles  of the $k$-wiggle decomposition are all the same length. Let $C$ be a cycle, $S \subseteq V(C)$.  We say the pair $(C,S)$ is \textit{distance regular} if, when following the cycle in a given direction, every path between consecutive elements of $S$ has the same length.

\begin{prop}\label{2even}
Let $C$ and $C'$ be cycles, $S \subseteq V(C)$, where $(C,S)$ is distance regular, and $S' \subseteq V(C')$.  Assume that the underlying torus of $(C,S) \anc (C',S')$ allows the $k$-wiggle decomposition.  Then the $k$-wiggle decomposition on $(C,S) \anc (C',S')$ yields $k$ cycles of the same length. 
\end{prop}

See Figure \ref{2wiggle} for an illustration with $k=2$. In the figure, $|S|=4$, $|S'|=8$, and $(C,S)$ is distance regular as each path between consecutive elements of $S$ has length 2.  Each cycle in the 2-wiggle decomposition has 52 total edges: 20 horizontal edges and 32 vertical edges.

\begin{proof}
Let $C_1, \ldots, C_{k}$ be the the cycles in the $k$-wiggle decomposition. For $1 \le \ell \le k$, $C_\ell$  has $|S|/k$ edges in each column, and thus each cycle has the same number of horizontal edges in the subdivided torus. For $1 \le \ell \le k$, $C_\ell$ has $|S'|-k+1$ edges in each row whose edges were obtained in a subdivision of  the edges from a row of index congruent to $\ell \pmod{k}$ in the underlying torus, and $1$ edge in each other row.  Since the union of edges in all rows form vertical copies of $C$ and $(C,S)$ is distance regular, all $k$ cycles have the same number of vertical edges.  Thus every cycle has the same length.
\end{proof}

The conclusion of Proposition~\ref{2even} also holds under the weaker assumption that the sum of the lengths of every $k$th path in $(C,S)$ is identical. For example, $(C,S)$ would meet this condition when $k=3$ if the consecutive path lengths were ${\bf 1}, 1, 2, {\bf 1}, 3, 2, {\bf 3}, 1, 1$, since the sum of the length of every third path is $5$, see for example the bold numbers giving the sum $1+1+3$.  However we will not need this generality so we use the simpler distance regular condition.

\subsection{Splittable decompositions}

In this section we define splittable decompositions, and prove some related properties about $k$-wiggle decompositions of subdivided tori.

A set of graphs $\{G_1, \ldots,  G_a\}$ forms a {\it splittable decomposition} of a graph $G$ if 
it is a decomposition of $G$ and for $i=1, \ldots, a$, there are pairwise disjoint sets $S_i \subseteq V(G_i)$ with $|S_1| = |S_2| = \cdots = |S_a| \ge 2$, whose union is $V(G)$.  We refer to the sets $S_1, \ldots, S_a$ as {\it representing sets} of the decomposition. The term splittable comes from the fact that the vertices of $G$ can be split evenly among the graphs in the decomposition into their representing sets.

For $a, m\ge 1$, if the set of graphs $\{G_1, \ldots, G_{am}\}$ is a decomposition of a graph $G$, we say it forms an {\it $a$-splittable decomposition} of $G$ if the set $\{G_1, \ldots, G_{am}\}$ can be partitioned into $m$ pairwise disjoint subsets $\F_1, \ldots, \F_m$,  each containing $a$ graphs, such that the graphs in each $\F_i$, $i=1, \ldots, m$ form a splittable decomposition of a spanning subgraph of $G$. We call these $\F_i$ the \textit{splitting sets} of the decomposition.  An $a$-splittable decomposition of $G$ is called an \textit{$(a,b)$-splittable decomposition} if 
each $\F_i$ can be partitioned into subsets $\F_{i,1}, \ldots, \F_{i,a/b}$, each of cardinality $b$, where the graphs in $\F_{i,j}$  are pairwise vertex disjoint and span $V(G)$.  We call these $\F_{i,j}$ the \textit{splitting subsets} of the decomposition.
Note that if all of the graphs in an $(a,b)$-splittable decomposition have the same number of vertices $v$, then $b= |V(G)|/v$.  

Note that a decomposition $\{G_1,\ldots, G_a\}$ of $G$ is $1$-splittable  if and only if each graph $G_i$ is a spanning subgraph of $G$. We call such a decomposition a \textit{spanning decomposition}, and in the case of a cycle decomposition, we call it a \textit{Hamiltonian decomposition}, since every graph in the decomposition is a Hamiltonian cycle. Note that for any $a$, an $(a,1)$-splittable decomposition is also a spanning decomposition and an $a$-splittable decomposition $\{G_1, \ldots, G_a\}$ of $G$ with $a$ graphs is simply a splittable decomposition. We shall use each notion when convenient.

An $a$-splittable (resp. $(a,b)$-splittable) cycle decomposition of a graph $G$ is called \textit{$a$-DR-splittable} (resp. \textit{$(a,b)$-DR-splittable})   if in addition to the other conditions, for all cycles $C$ in the decomposition, if $S$ is the representing set for $C$, then $(C,S)$ is distance regular.

\begin{prop}\label{kwiggle}
The decomposition into cycles produced by the $k$-wiggle decomposition on a torus is $k$-DR-splittable. If $k$ is even, the decomposition is also $k/2$-DR-splittable.
\end{prop}

\begin{proof}
Let $C_1, \ldots, C_{k}$ be the the cycles in the $k$-wiggle decomposition of a torus $T$.  For $1 \le \ell \le k$ we need to find representing sets $S_\ell \subseteq V(C_\ell)$, all of the same cardinality, partitioning $V(T)$ and splitting the cycles into paths of equal length. Let $S_1$ be the set consisting of every other vertex encountered as $C_1$ is being traversed in a given direction.  For $2 \le \ell \le k$, let $S_\ell$ be the vertical translation of $S_1$ by $\ell-1$, i.e., the vertex $(i,j) \in S_1$ if and only if the vertex $(i +\ell-1, j) \in S_\ell$. Note that every $k$th vertex in each vertical cycle is part of a given $S_\ell$, so these sets partition $V(T)$ and have the same cardinality. Further, since the cycles $C_1, \ldots, C_k$ are all vertical translations of each other, for all $\ell$,  $S_\ell$ is the set consisting of every other vertex of $V(C_\ell)$ encountered as $C_\ell$ is being traversed in a given direction. 
Thus every path in $C_\ell$ between consecutive elements of $S_\ell$ has length 2, and $(C_\ell,S_\ell)$ is distance regular, see Figure \ref{wiggle-DR} (left).   

Let $k$ be even. To show that the decomposition is $k/2$-DR-splittable, we need to partition the cycles into two splitting sets of $k/2$ cycles each and for each splitting set find representing sets of vertices in each cycle of the same cardinality, partitioning $V(T)$ and dividing the cycles into paths of equal length.
Let the first splitting set $\F_1$ contain the $k/2$ cycles with odd indices, $\F_1 = \{C_1, C_3, \ldots, C_{k-1}\}$, and the second splitting set $\F_2$ contain the $k/2$ cycles with even indices, $\F_2 = \{C_2, C_4, \ldots, C_k\}$. For each cycle $C_\ell$, let $S_\ell=V(C_\ell)$.  Then since every vertex in the torus is contained in one even-indexed cycle and one odd-indexed cycle, the representing sets in each splitting set partition $V(T)$ and every path in $C_\ell$ between consecutive elements of $S_\ell$ has length 1, see Figure \ref{wiggle-DR} (right).  
\end{proof}

Note that if the decomposition of a torus obtained by the $k$-wiggle decomposition is $a$-splittable, then $a$ must be $k$ or $k/2$: Each cycle covers exactly $2/k$ proportion of the vertices in the torus. Thus at least half of the $k$ cycles are required to cover all the vertices in the torus, which implies at least half of the cycles must be in each splitting set.

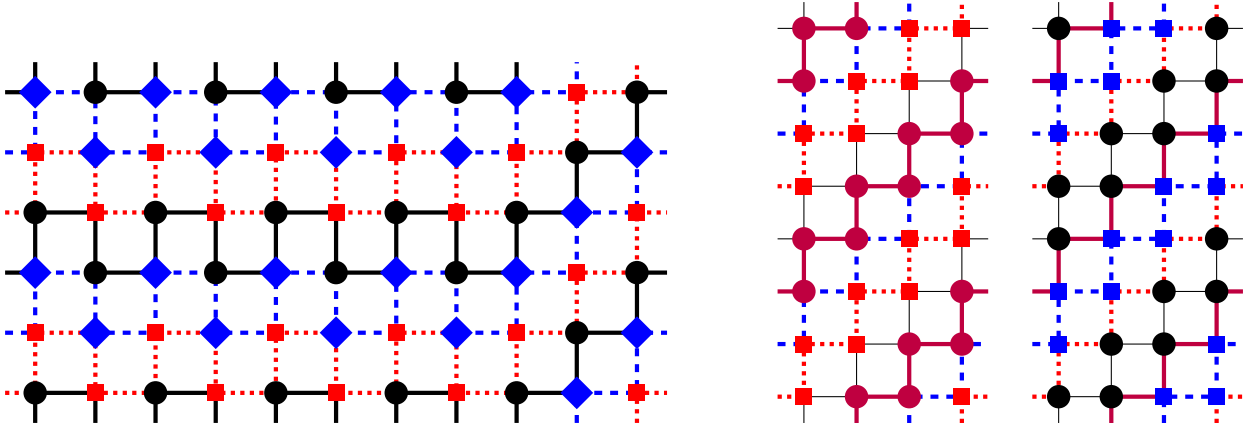
\begin{figure}
\begin{center}
\begin{tikzpicture}[scale=.8]  


    \draw[ultra thick, red, dotted]  (3.5,0)-- (4,0) -- (4,1) -- (5,1) -- (5,0) -- (6,0) -- (6,1) -- (7,1) -- (7,0) -- (8,0) -- (8,1) -- (9,1) -- (9,0) -- (10,0) -- (10,1) -- (11,1) -- (11,0) -- (12,0) --(12,1) -- (13,1) -- (13,2) -- (14,2) -- (14,3) -- (14.5,3);
    \draw[ultra thick, red, dotted]  (3.5,3) -- (4,3) -- (4,4) -- (5,4) -- (5,3) -- (6,3) -- (6,4) -- (7,4) -- (7,3) -- (8,3) -- (8,4) -- (9,4) -- (9,3) -- (10,3) -- (10,4) -- (11,4) -- (11,3) -- (12,3) --(12,4) -- (13,4) -- (13,5) -- (14,5) -- (14,5) -- (14,5.5);
    \draw[ultra thick, red, dotted] (14,-.5) -- (14,0) -- (14.5,0);

    \draw[ultra thick, blue, dashed] (3.5,1) -- (4,1) -- (4,2) -- (5,2) -- (5,1) -- (6,1) -- (6,2) -- (7,2) -- (7,1) -- (8,1) -- (8,2) -- (9,2) -- (9,1) -- (10,1) -- (10,2) -- (11,2) -- (11,1) -- (12,1) --(12,2) -- (13,2) -- (13,3) -- (14,3) -- (14,4) -- (14.5,4);
    \draw[ultra thick, blue, dashed] (3.5,4) -- (4,4) -- (4,5) -- (5,5) -- (5,4) -- (6,4) -- (6,5) -- (7,5) -- (7,4) -- (8,4) -- (8,5) -- (9,5) -- (9,4) -- (10,4) -- (10,5) -- (11,5) -- (11,4) -- (12,4) --(12,5) -- (13,5) -- (13,5.5);
    \draw[ultra thick, blue, dashed] (13,-.5) -- (13,0) -- (14,0) --(14,1) -- (14.5,1);
    
       \draw[ultra thick]  (3.5,2)-- (4,2) -- (4,3) -- (5,3) -- (5,2) -- (6,2) -- (6,3) -- (7,3) -- (7,2) -- (8,2) -- (8,3) -- (9,3) -- (9,2) -- (10,2) -- (10,3) -- (11,3) -- (11,2) -- (12,2) --(12,3) -- (13,3) -- (13,4) -- (14,4) -- (14,5) -- (14.5,5);
              \draw[ultra thick]  (4,-.5)-- (4,0) -- (5,0) -- (5,-.5) ;
              \draw[ultra thick]  (6,-.5)-- (6,0) -- (7,0) -- (7,-.5) ;
              \draw[ultra thick]  (8,-.5)-- (8,0) -- (9,0) -- (9,-.5) ;
              \draw[ultra thick]  (10,-.5)-- (10,0) -- (11,0) -- (11,-.5) ;
    \draw[ultra thick] (12,-.5) -- (12,0) -- (13,0) -- (13,1) -- (14,1) -- (14,2) -- (14.5,2);
    
                  \draw[ultra thick]  (3.5,5)-- (4,5) -- (4,5.5);
              \draw[ultra thick]  (5,5.5)-- (5,5) -- (6,5) -- (6,5.5) ;
              \draw[ultra thick]  (7,5.5)-- (7,5) -- (8,5) -- (8,5.5) ;
        \draw[ultra thick]  (9,5.5)-- (9,5) -- (10,5) -- (10,5.5) ;
                \draw[ultra thick]  (11,5.5) -- (11,5) -- (12,5) -- (12,5.5) ;
              

    \foreach \x in {5, 7, 9, 11, 14}{
    \foreach \y in {0, 3}{
      \node[style={rectangle ,fill=red},scale=.8] at (\x,\y){};
}
}

    \foreach \x in {4, 6, 8, 10, 12}{
    \foreach \y in {1, 4}{
      \node[style={rectangle,fill=red},scale=.8] at (\x,\y){};
}
}

    \foreach \x in {13}{
    \foreach \y in {2, 5}{
      \node[style={rectangle,fill=red},scale=.8] at (\x,\y){};
}
}

    \foreach \x in {5, 7, 9, 11, 14}{
    \foreach \y in {1, 4}{
      \node[style={diamond,fill=blue},scale=.8] at (\x,\y){};
}
}

    \foreach \x in {4, 6, 8, 10, 12}{
    \foreach \y in {2, 5}{
      \node[style={diamond, fill=blue},scale=.8] at (\x,\y){};
}
}

    \foreach \x in {13}{
    \foreach \y in {0, 3}{
      \node[style={diamond, fill=blue},scale=.8] at (\x,\y){};
}
}

    \foreach \x in {5, 7, 9, 11, 14}{
    \foreach \y in {2, 5}{
      \node[style={circle,fill=black},scale=.8] at (\x,\y){};
}
}

    \foreach \x in {4, 6, 8, 10, 12}{
    \foreach \y in {0, 3}{
      \node[style={circle,fill=black},scale=.8] at (\x,\y){};
}
}

    \foreach \x in {13}{
    \foreach \y in {1, 4}{
      \node[style={circle,fill=black},scale=.8] at (\x,\y){};
}
}

\end{tikzpicture}
\hfill
\begin{tikzpicture}[scale=.7]  


    \draw[ultra thick, red, dotted] (-.5,0) -- (0,0) -- (0,1) -- (1,1) -- (1,2) -- (2,2) -- (2,3) -- (3,3) -- (3,4) -- (3.5,4);
    \draw[ultra thick, red, dotted]  (-.5,4) -- (0,4) -- (0,5) -- (1,5) -- (1,6) -- (2,6) -- (2,7) -- (3,7) -- (3,7.5);
   \draw[ultra thick, red, dotted] (3,-.5) -- (3,0) -- (3.5,0);

    \draw[ultra thick, blue, dashed] (-.5,1) -- (0,1) -- (0,2) -- (1,2) -- (1,3) -- (2,3) -- (2,4) -- (3,4) -- (3,5) -- (3.5,5);
    \draw[ultra thick, blue, dashed]  (-.5,5) -- (0,5) -- (0,6) -- (1,6) -- (1,7) -- (2,7) -- (2,7.5);
   \draw[ultra thick, blue, dashed] (2,-.5) -- (2,0) -- (3,0) -- (3,1) -- (3.5,1);

    \draw[ultra thick, purple] (-.5,2) -- (0,2) -- (0,3) -- (1,3) -- (1,4) -- (2,4) -- (2,5) -- (3,5) -- (3,6) -- (3.5,6);
    \draw[ultra thick, purple]  (-.5,6) -- (0,6) -- (0,7) -- (1,7) -- (1,7.5);
   \draw[ultra thick, purple] (1,-.5) -- (1,0) -- (2,0) -- (2,1) -- (3,1) -- (3,2) --(3.5,2);

    \draw (-.5,3) -- (0,3) -- (0,4) -- (1,4) -- (1,5) -- (2,5) -- (2,6) -- (3,6) -- (3,7) -- (3.5,7);
    \draw  (-.5,7) -- (0,7) -- (0,7.5);
   \draw (0,-.5) -- (0,0) -- (1,0) -- (1,1) -- (2,1) -- (2,2) --(3,2) -- (3,3) -- (3.5,3);

    \foreach \x in {0,...,2}{
      \node[style={rectangle,fill=red},scale=.8] at (\x,\x){};
      \node[style={rectangle,fill=red},scale=.8] at (\x,\x+1){};
      \node[style={rectangle,fill=red},scale=.8] at (\x,\x+4){};
      \node[style={rectangle,fill=red},scale=.8] at (\x,\x+5){};
}

    \foreach \y in {0,3,4,7}{
      \node[style={rectangle,fill=red},scale=.8] at (3,\y){};
}

    \foreach \y in {2,3,6,7}{
      \node[style={circle,fill=purple},scale=.8] at (0,\y){};
}
    \foreach \y in {0, 3, 4, 7}{
      \node[style={circle,fill=purple},scale=.8] at (1,\y){};
}
    \foreach \y in {0, 1, 4, 5}{
      \node[style={circle,fill=purple},scale=.8] at (2,\y){};
}

    \foreach \y in {1, 2, 5, 6}{
      \node[style={circle,fill=purple},scale=.8] at (3,\y){};
}

\end{tikzpicture}
\hspace{.1in}
\begin{tikzpicture}[scale=.7]  

    \draw[ultra thick, red, dotted] (-.5,0) -- (0,0) -- (0,1) -- (1,1) -- (1,2) -- (2,2) -- (2,3) -- (3,3) -- (3,4) -- (3.5,4);
    \draw[ultra thick, red, dotted]  (-.5,4) -- (0,4) -- (0,5) -- (1,5) -- (1,6) -- (2,6) -- (2,7) -- (3,7) -- (3,7.5);
   \draw[ultra thick, red, dotted] (3,-.5) -- (3,0) -- (3.5,0);

    \draw[ultra thick, blue, dashed] (-.5,1) -- (0,1) -- (0,2) -- (1,2) -- (1,3) -- (2,3) -- (2,4) -- (3,4) -- (3,5) -- (3.5,5);
    \draw[ultra thick, blue, dashed]  (-.5,5) -- (0,5) -- (0,6) -- (1,6) -- (1,7) -- (2,7) -- (2,7.5);
   \draw[ultra thick, blue, dashed] (2,-.5) -- (2,0) -- (3,0) -- (3,1) -- (3.5,1);

    \draw[ultra thick, purple] (-.5,2) -- (0,2) -- (0,3) -- (1,3) -- (1,4) -- (2,4) -- (2,5) -- (3,5) -- (3,6) -- (3.5,6);
    \draw[ultra thick, purple]  (-.5,6) -- (0,6) -- (0,7) -- (1,7) -- (1,7.5);
   \draw[ultra thick, purple] (1,-.5) -- (1,0) -- (2,0) -- (2,1) -- (3,1) -- (3,2) --(3.5,2);

    \draw (-.5,3) -- (0,3) -- (0,4) -- (1,4) -- (1,5) -- (2,5) -- (2,6) -- (3,6) -- (3,7) -- (3.5,7);
    \draw  (-.5,7) -- (0,7) -- (0,7.5);
   \draw (0,-.5) -- (0,0) -- (1,0) -- (1,1) -- (2,1) -- (2,2) --(3,2) -- (3,3) -- (3.5,3);


    \foreach \y in {1, 2, 5, 6}{
      \node[style={rectangle,fill=blue},scale=.8] at (0,\y){};
}
    \foreach \y in {2, 3, 6, 7}{
      \node[style={rectangle,fill=blue},scale=.8] at (1,\y){};
}
    \foreach \y in {0, 3, 4, 7}{
      \node[style={rectangle,fill=blue},scale=.8] at (2,\y){};
}
    \foreach \y in {0, 1, 4, 5}{
      \node[style={rectangle,fill=blue},scale=.8] at (3,\y){};
}

    \foreach \y in {0, 3, 4, 7}{
      \node[style={circle,fill=black},scale=.8] at (0,\y){};
}
    \foreach \y in {0, 1, 4, 5}{
      \node[style={circle,fill=black},scale=.8] at (1,\y){};
}
    \foreach \y in {1, 2, 5, 6}{
      \node[style={circle,fill=black},scale=.8] at (2,\y){};
}
    \foreach \y in {2, 3, 6, 7}{
      \node[style={circle,fill=black},scale=.8] at (3,\y){};
}

\end{tikzpicture}

\end{center}
\caption{Left: A 3-DR-splittable cycle decomposition of $C_6 \sq C_{11}$ produced by the 3-wiggle decomposition. There is one splitting set, $\F_1$, containing all three cycles (dotted red, dashed blue, and black).  The red square vertices, blue diamond vertices, and black circular vertices are the representing sets for the cycles with the corresponding color.  Right: A 2-DR-splittable cycle decomposition of $C_8 \sq C_4$ produced by the 4-wiggle decomposition.  There are two splitting sets: $\F_1$ contains the dotted red and thick purple cycles, and the representing sets for these cycles are the square red and circular purple vertices illustrated on the left, while $\F_2$ contains the dashed blue and thin black cycles, and the representing sets for these cycles are the square blue and circular black vertices illustrated on the right.}\label{wiggle-DR}
\end{figure}

\begin{prop}\label{dregprod}
Suppose the torus $C \sq C'$ allows the $k$-wiggle decomposition and there is a set $S' \subseteq V(C')$ such that $(C',S')$ is distance regular.  Then there are sets $S_1, \ldots, S_k$, each of the same cardinality, partitioning $V(C) \times S'$ such that for the cycles $C_1, \ldots, C_k$ produced by the $k$-wiggle decomposition on $C \sq C'$, for $1 \le \ell \le k$, $(C_\ell, S_\ell)$ is distance regular.
\end{prop}

\begin{proof}
Let $S_1$ be the set consisting of every other vertex of $(V(C) \times S')\cap V(C_1)$ encountered as $C_1$ is being traversed in a given direction.  For $2 \le \ell \le k$, let $S_\ell$ be the vertical translation of $S_1$ by $\ell-1$, i.e., the vertex $(i,j) \in S_1$ if and only if the vertex $(i+\ell-1, j) \in S_\ell$. Note that every $k$th vertex in each vertical cycle is part of a given $S_\ell$, so these sets partition $V(C) \times S'$ and have the same cardinality, and for all $\ell$,  $S_\ell$ is the set consisting of every other vertex of $(V(C) \times S')\cap V(C_\ell)$ encountered as $C_\ell$ is being traversed in a given direction. 
Thus every path in $C_\ell$ between consecutive elements of $S_\ell$ is twice as long as the corresponding path in the horizontal cycle $C'$ between consecutive elements of $S'$, and $(C_\ell,S_\ell)$ is distance regular if and only if $(C', S')$ is. See Figure~\ref{prod}.
\end{proof}

\begin{figure}
\begin{center}

\begin{tikzpicture}[scale=.8]  


    \draw[ultra thick, dotted, red] (-.5,0) -- (0,0) -- (0,1) -- (1,1) -- (1,0) -- (2,0) -- (2,1) -- (3,1) -- (3,0) -- (4,0) -- (4,1) -- (5,1) -- (5,0) -- (6,0) -- (6,1) -- (7,1) -- (7,0) -- (8,0) -- (8,1) -- (9,1) -- (9,0) -- (10,0) -- (10,1) -- (11,1) -- (11,0) -- (12,0) --(12,1) -- (13,1) -- (13,2) -- (14,2) -- (14,3) -- (14.5,3);
    \draw[ultra thick, dotted, red] (-.5,3) -- (0,3) -- (0,4) -- (1,4) -- (1,3) -- (2,3) -- (2,4) -- (3,4) -- (3,3) -- (4,3) -- (4,4) -- (5,4) -- (5,3) -- (6,3) -- (6,4) -- (7,4) -- (7,3) -- (8,3) -- (8,4) -- (9,4) -- (9,3) -- (10,3) -- (10,4) -- (11,4) -- (11,3) -- (12,3) --(12,4) -- (13,4) -- (13,5) -- (14,5) -- (14,5) -- (14,5.5);
    \draw[ultra thick, dotted, red] (14,-.5) -- (14,0) -- (14.5,0);

    \draw[ultra thick, dashed, blue] (-.5,1) -- (0,1) -- (0,2) -- (1,2) -- (1,1) -- (2,1) -- (2,2) -- (3,2) -- (3,1) -- (4,1) -- (4,2) -- (5,2) -- (5,1) -- (6,1) -- (6,2) -- (7,2) -- (7,1) -- (8,1) -- (8,2) -- (9,2) -- (9,1) -- (10,1) -- (10,2) -- (11,2) -- (11,1) -- (12,1) --(12,2) -- (13,2) -- (13,3) -- (14,3) -- (14,4) -- (14.5,4);
    \draw[ultra thick, dashed , blue] (-.5,4) -- (0,4) -- (0,5) -- (1,5) -- (1,4) -- (2,4) -- (2,5) -- (3,5) -- (3,4) -- (4,4) -- (4,5) -- (5,5) -- (5,4) -- (6,4) -- (6,5) -- (7,5) -- (7,4) -- (8,4) -- (8,5) -- (9,5) -- (9,4) -- (10,4) -- (10,5) -- (11,5) -- (11,4) -- (12,4) --(12,5) -- (13,5) -- (13,5.5);
    \draw[ultra thick, dashed, blue] (13,-.5) -- (13,0) -- (14,0) --(14,1) -- (14.5,1);
    
        \draw[ultra thick] (-.5,2) -- (0,2) -- (0,3) -- (1,3) -- (1,2) -- (2,2) -- (2,3) -- (3,3) -- (3,2) -- (4,2) -- (4,3) -- (5,3) -- (5,2) -- (6,2) -- (6,3) -- (7,3) -- (7,2) -- (8,2) -- (8,3) -- (9,3) -- (9,2) -- (10,2) -- (10,3) -- (11,3) -- (11,2) -- (12,2) --(12,3) -- (13,3) -- (13,4) -- (14,4) -- (14,5) -- (14.5,5);

    \draw[ultra thick] (0,-.5) -- (0,0) -- (1,0) --(1, -.5) ;
    \draw[ultra thick] (2,-.5) -- (2,0) -- (3,0) --(3, -.5) ;
        \draw[ultra thick] (4,-.5) -- (4,0) -- (5,0) --(5, -.5) ;
            \draw[ultra thick] (6,-.5) -- (6,0) -- (7,0) --(7, -.5) ;
                \draw[ultra thick] (8,-.5) -- (8,0) -- (9,0) --(9, -.5) ;
                \draw[ultra thick] (10,-.5) -- (10,0) -- (11,0) --(11, -.5) ;
                \draw[ultra thick] (12,-.5) -- (12,0) -- (13,0) --(13,1) --(14, 1) -- (14,2) -- (14.5,2) ;

 \draw[ultra thick] (-.5,5) -- (0,5) -- (0,5.5);
  \draw[ultra thick] (1,5.5) -- (1,5) -- (2,5) -- (2,5.5);
   \draw[ultra thick] (3,5.5) -- (3,5) -- (4,5) -- (4,5.5);
   \draw[ultra thick] (5,5.5) -- (5,5) -- (6,5) -- (6,5.5);
   \draw[ultra thick] (7,5.5) -- (7,5) -- (8,5) -- (8,5.5);
     \draw[ultra thick] (9,5.5) -- (9,5) -- (10,5) -- (10,5.5);
     \draw[ultra thick] (11,5.5) -- (11,5) -- (12,5) -- (12,5.5);

    \foreach \x in { 0, 1, 2, 4 ,5 ,6, 7, 9, 10, 11, 12,  14}{
    \foreach \y in {0, ..., 5}{
      \node[style={circle, draw, fill=white},scale=.8] at (\x,\y){};
}
}

    \foreach \x in {3}{
    \foreach \y in {0, 3}{
      \node[style={rectangle,fill=red},scale=.8] at (\x,\y){};
}
}

    \foreach \x in {8}{
    \foreach \y in {1, 4}{
      \node[style={rectangle,fill=red},scale=.8] at (\x,\y){};
}
}

    \foreach \x in {13}{
    \foreach \y in {2, 5}{
      \node[style={rectangle,fill=red},scale=.8] at (\x,\y){};
}
}

    \foreach \x in {3}{
    \foreach \y in {1, 4}{
      \node[style={diamond,fill=blue},scale=.8] at (\x,\y){};
}
}

    \foreach \x in {8}{
    \foreach \y in {2, 5}{
      \node[style={diamond,fill=blue},scale=.8] at (\x,\y){};
}
}

    \foreach \x in {13}{
    \foreach \y in {0, 3}{
      \node[style={diamond,fill=blue},scale=.8] at (\x,\y){};
}
}

    \foreach \x in {3}{
    \foreach \y in {2, 5}{
      \node[style={circle,fill=black},scale=.8] at (\x,\y){};
}
}

    \foreach \x in {8}{
    \foreach \y in {0, 3}{
      \node[style={circle,fill=black},scale=.8] at (\x,\y){};
}
}

    \foreach \x in {13}{
    \foreach \y in {1, 4}{
      \node[style={circle,fill=black},scale=.8] at (\x,\y){};
}
}

    \draw[ultra thick] (-.5,-2) -- (14.5,-2);

    \foreach \x in {0,...,14}{
      \node[style={circle,draw, fill=white},scale=.8] at (\x,-2){};
}

    \foreach \x in {3, 8, 13}{
      \node[style={circle,fill=black},scale=.8] at (\x,-2){};
}

      \node (1) at (7,-2) [label=below:{$(C', S')$}]{};

\end{tikzpicture}

\end{center}
\caption{An example for Proposition~\ref{dregprod}:  The 3-wiggle is applied to $C \sq C' = C_6 \sq C_{15}$,  to produce a decomposition into three cycles $C_1$, $C_2$, and $C_3$ (dotted red, dashed blue, and black). In this example $S'=\{3, 8, 13\}$, and the vertices in $C \times S'$ are partitioned into sets $S_1$, $S_2$, and $S_3$ (red squares, blue diamonds, black circles). Note that each path in $C'$ between consecutive elements of $S'$ has length 5, and each path in $C_i$ between consecutive vertices in $S_i$ has length 10.}\label{prod}
\end{figure}
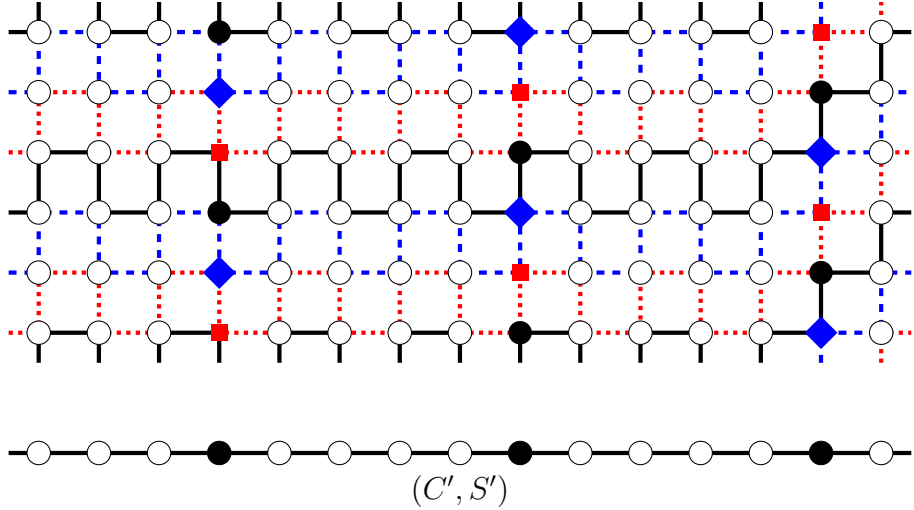


\begin{prop}\label{subsplit}
Suppose the subdivided torus $(C,S) \anc (C',S')$ allows the $k$-wiggle decomposition, $(C,S)$ is distance regular, and  $C_1, \ldots, C_k$ are the cycles produced by the $k$-wiggle decomposition on $(C,S) \anc (C',S')$. Then there are sets $S_1, \ldots, S_k$, each of the same cardinality, partitioning $V(C) \times S'$ such that for $1 \le \ell \le k$, $S_\ell \subseteq V(C_\ell)$. 
\end{prop}

\begin{proof}
All degree two vertices in the subdivided torus that are in $V(C) \times S'$ lie on only one $C_\ell$, and so go in the corresponding $S_\ell$. The fact that $(C,S)$ is distance regular and each cycle contains every $k$th path in each vertical cycle guarantees that there are the same number of each of these in each $S_\ell$. It remains to assign the degree four vertices in $V(C) \times S'$, so we ignore the degree two vertices, and consider the underlying torus, with vertex set $S \times S'$. We assign the vertices of the underlying torus to $S_1, \ldots, S_k$ in the alternating pattern of Propositions~\ref{kwiggle} and \ref{dregprod}, so that every other degree 4 vertex on a given cycle $C_\ell$ is in its corresponding $S_\ell$. See Figure~\ref{2wiggle}.
\end{proof}

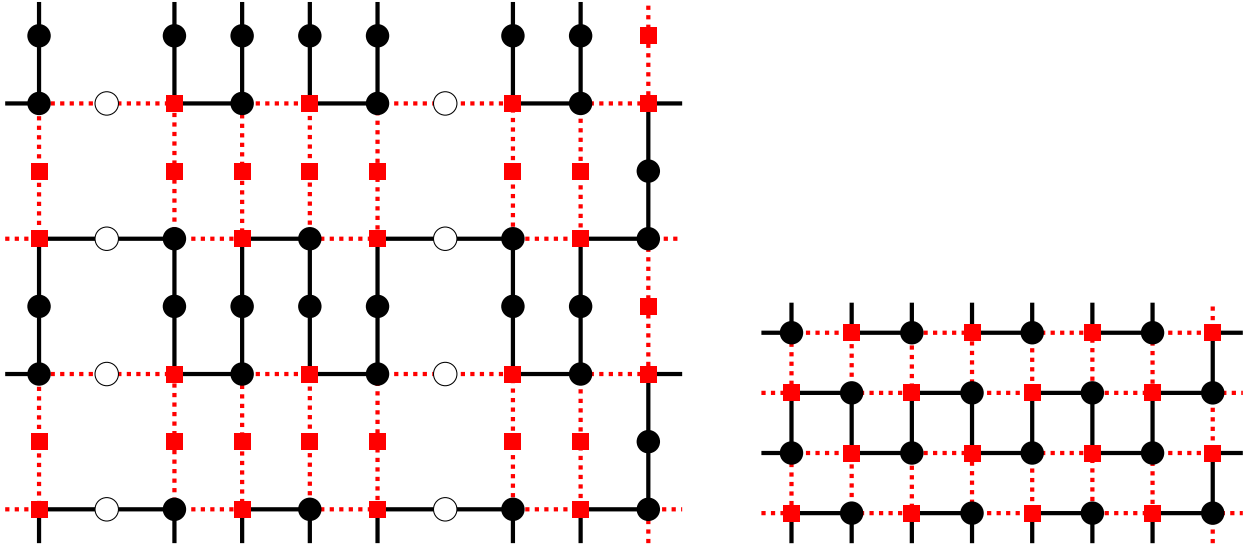
\begin{figure}
\begin{center}
\begin{tikzpicture}[scale=.9]  


    \draw[ultra thick, dotted, red] (-.5,0) -- (0,0) -- (0,2) -- (2,2) -- (2,0) -- (3,0) -- (3,2) -- (4,2) -- (4,0) -- (5,0) -- (5,2) -- (7,2) -- (7,0) -- (8,0) -- (8,2) -- (9,2) -- (9,4) --(9.5,4);
    \draw[ultra thick, dotted, red] (-.5,4) -- (0,4) -- (0,6) -- (2,6) -- (2,4) -- (3,4) -- (3,6) -- (4,6) -- (4,4) -- (5,4) -- (5,6) -- (7,6) -- (7,4) -- (8,4) -- (8,6) -- (9,6) -- (9,7.5); 
   \draw[ultra thick, dotted, red] (9,-.5) -- (9,0) -- (9.5,0);

       \draw[ultra thick] (-.5,2) -- (0,2) -- (0,4) -- (2,4) -- (2,2) -- (3,2) -- (3,4) -- (4,4) -- (4,2) -- (5,2) -- (5,4) -- (7,4) -- (7,2) -- (8,2) -- (8,4) -- (9,4) -- (9,6) --(9.5,6);
   \draw[ultra thick] (-.5,6) -- (0,6) -- (0, 7.5);
   \draw[ultra thick] (2,7.5) -- (2,6) -- (3,6) -- (3,7.5) ;
   \draw[ultra thick] (4,7.5) -- (4,6) -- (5,6) -- (5,7.5) ;
   \draw[ultra thick] (7,7.5) -- (7,6) -- (8,6) -- (8,7.5) ;
   \draw[ultra thick] (0,-.5) -- (0,0) -- (2,0) -- (2,-.5) ;
   \draw[ultra thick] (3,-.5) -- (3,0) -- (4,0) -- (4,-.5) ;
   \draw[ultra thick] (5,-.5) -- (5,0) -- (7,0) -- (7,-.5) ;
   \draw[ultra thick] (8,-.5) -- (8,0) -- (9,0) -- (9,2) --(9.5,2);

    \foreach \x in {0,  3,  5, 8}{
    \foreach \y in {2, 3, 6, 7}{
      \node[style={circle,fill=black},scale=.8] at (\x,\y){};
}
}

    \foreach \x in {2, 4, 7}{
    \foreach \y in {0, 3, 4, 7}{
      \node[style={circle,fill=black},scale=.8] at (\x,\y){};
}
}

    \foreach \x in {9}{
    \foreach \y in {0, 1, 4, 5}{
      \node[style={circle,fill=black},scale=.8] at (\x,\y){};
}
}

    \foreach \x in {0,  3,  5,  8}{
    \foreach \y in {0, 1, 4, 5}{
      \node[style={rectangle,fill=red},scale=.8] at (\x,\y){};
}
}

    \foreach \x in {2, 4, 7}{
    \foreach \y in {1, 2, 5, 6}{
      \node[style={rectangle,fill=red},scale=.8] at (\x,\y){};
}
}

    \foreach \x in {9}{
    \foreach \y in {2,3, 6, 7}{
      \node[style={rectangle,fill=red},scale=.8] at (\x,\y){};
}
}

    \foreach \x in {1, 6}{
    \foreach \y in {0, 2, 4, 6}{
      \node[style={circle,draw, fill=white},scale=.8] at (\x,\y){};
}
}

\end{tikzpicture}
\hfill
\begin{tikzpicture}[scale=.8]  

    \draw[ultra thick, dotted, red] (-.5,0) -- (0,0) -- (0,1) -- (1,1) -- (1,0) -- (2,0) -- (2,1) -- (3,1) -- (3,0) -- (4,0) -- (4,1) -- (5,1) -- (5,0) -- (6,0) -- (6,1) --(7,1) -- (7,2) --(7.5,2);
    \draw[ultra thick, dotted, red] (-.5,2) -- (0,2) -- (0,3) -- (1,3) -- (1,2) -- (2,2) -- (2,3) -- (3,3) -- (3,2) -- (4,2) -- (4,3) -- (5,3) -- (5,2) -- (6,2) -- (6,3) --(7,3) -- (7,3.5);
   \draw[ultra thick, dotted, red] (7,-.5) -- (7,0) -- (7.5,0);

    \draw[ultra thick] (-.5,1) -- (0,1) -- (0,2) -- (1,2) -- (1,1) -- (2,1) -- (2,2) -- (3,2) -- (3,1) -- (4,1) -- (4,2) -- (5,2) -- (5,1) -- (6,1) -- (6,2) --(7,2) -- (7,3) --(7.5,3);
   \draw[ultra thick] (6,-.5) -- (6,0) -- (7,0) -- (7,1) -- (7.5,1);
   \draw[ultra thick] (0,-.5) -- (0,0) -- (1,0) -- (1,-.5);
   \draw[ultra thick] (2,-.5) -- (2,0) -- (3,0) -- (3,-.5);
   \draw[ultra thick] (4,-.5) -- (4,0) -- (5,0) -- (5,-.5);

   \draw[ultra thick] (-.5,3) -- (0,3) -- (0,3.5);
   \draw[ultra thick] (1,3.5) -- (1,3) -- (2,3) -- (2,3.5);
   \draw[ultra thick] (3,3.5) -- (3,3) -- (4,3) -- (4,3.5);
   \draw[ultra thick] (5,3.5) -- (5,3) -- (6,3) -- (6,3.5);

    \foreach \x in {0, 2, 4, 6}{
    \foreach \y in {0, 2}{
      \node[style={rectangle,fill=red},scale=.8] at (\x,\y){};
}
}

    \foreach \x in {0, 2, 4, 6}{
    \foreach \y in {1, 3}{
      \node[style={circle,fill=black},scale=.8] at (\x,\y){};
}
}

    \foreach \x in {1, 3, 5, 7}{
    \foreach \y in {0, 2}{
      \node[style={circle,fill=black},scale=.8] at (\x,\y){};
}
}

    \foreach \x in {1, 3, 5, 7}{
    \foreach \y in {1, 3}{
      \node[style={rectangle,fill=red},scale=.8] at (\x,\y){};
}
}
\end{tikzpicture}

\end{center}
\caption{Left: A $2$-wiggle decomposition of a subdivided torus $(C,S) \anc (C',S')$ into two cycles $C_1$ and $C_2$ (black and dotted red).   Note that since $(C,S)$ is distance regular (though $(C',S')$ is not) the cycles have the same length.  In the notation of Proposition~\ref{subsplit},  the black circular vertices are in $S_1$ and the red square vertices are in $S_2$.  Note that the vertices in $S_i$ are on the cycle $C_i$, and $S_1$ and $S_2$ partition the vertices of $C \sq S'$. Right: The 2-wiggle decomposition on the underlying torus, where every other vertex on a given cycle is in its representing set.
}\label{2wiggle}
\end{figure}

\section{Decompositions of Cartesian products of graphs}\label{combine}

The main result in this section is Lemma \ref{p7gen}, which will be the key  tool for inductively generating cycle decompositions on the hypercube. First we need two general statements about decompositions of Cartesian product graphs.

\begin{prop}\label{decomp}
 Let the graphs  $G_1, \ldots,  G_{a}$ form a splittable decomposition of $G$  with representing sets $S_1, \ldots, S_{a}$ and the graphs  $G_1' , \ldots, G_{b}'$  form a splittable decomposition of $G'$  with representing sets $S_1',  \ldots, S_{b}'$. Then 
\[G \sq G' = (G_1 \cup \cdots\cup G_{a})  \sq  (G_1' \cup \cdots\cup G_{b}') = \bigcup_{i=1}^{a} \bigcup_{j=1}^b (G_i,S_{i}) \anc (G_{j}',S_j'),\]
where the union of anchored products is pairwise edge-disjoint, i.e., a decomposition.
\end{prop}

\begin{proof}

We shall verify that every edge of  $G \sq G'= (G_1 \cup \cdots\cup G_{a})  \sq  (G_1' \cup \cdots\cup G_{b}')$  is accounted for exactly once in the union of anchored products. Let $e\in E(G \sq G')$, where without loss of generality $e= (u,v)(u',v)$ for $uu'\in E(G_i)$ and $v\in S_j'$. Then we see that $e\in E((G_i,S_{i}) \anc (G_{j}',S_j'))$.  Now, consider $e\in E((G_i,S_{i}) \anc (G_{j}',S_j'))$, then $e\in E(G_i\sq G_j') \subseteq E(G \sq G')$. 
 Finally, we need to check that no edge of $e$ belongs to two different anchored products $(G_i,S_{i}) \anc (G_{j}',S_j')$ and 
 $(G_q,S_{q}) \anc (G_{p}',S_p')$. Since these products are different, assume without loss of generality that $p\neq j$.
 Thus $S_p'\cap S_j' =\emptyset$. If $e\in E((G_i,S_{i}) \anc (G_{j}',S_j'))$, then $e= (u,v)(u',v)$ for $uu'\in E(G), v\in S_j'$ or 
 $e=(u,v)(u,v')$ for $u\in S_i$ and $vv'\in E(G_j')$. In the former case, $v\in S_j'$, thus $v\not\in S_p'$, so $e\not\in E(G_q,S_{q}) \anc (G_{p}',S_p')$. In the latter case $vv'\in E(G_j')$, thus, since $E(G_j')\cap E(G_p')=\emptyset$, we have that $vv'\not\in E(G_p')$.
 Thus $e\not\in (G_q,S_{q}) \anc (G_{p}',S_p')$.
\end{proof}

\begin{prop}\label{spanning}
Let graphs $G$ and $G'$  each have a decomposition into $a\ge 1$ spanning subgraphs, $G_1, \ldots, G_a$ and $G_1', \ldots, G_a'$, respectively. Then 
$$G \sq G' = (G_1 \cup \cdots\cup G_a) \sq  (G_1' \cup \cdots\cup G_a') = \bigcup_{i=1}^a G_i \sq G_i' ,$$
where the union is pairwise edge-disjoint, i.e., a decomposition.
\end{prop}

\begin{proof}
Consider an edge $e\in E(G\sq G')$. Then $e=(u,v)(u',v)$ for $uu'\in E(G_i), v\in V(G')$ or $e=(u,v)(u,v')$ for $u\in V(G), vv'\in E(G_i')$  for some $i=1, \ldots, a$. In both cases $e\in E(G_i\sq G_i')$.
Clearly any edge in $G_i\sq G_i'$ is in $G\sq G'$. 
Assume that there is an edge $e$, $e\in E(G_i\sq G_i')$, $e\in E(G_j\sq G_j')$, $i\neq j$. Without loss of generality $e=(u,v)(u',v)$.
Then $uu'\in E(G_i)\cap E(G_j)$, a contradiction.
\end{proof}

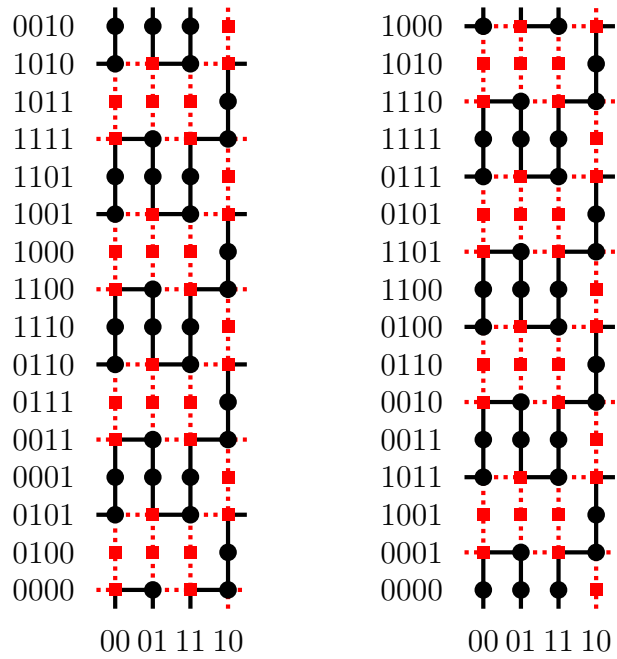
\begin{figure}
\begin{center}

\begin{tikzpicture}[scale=.5]  

    \draw[ultra thick, dotted, red] (-.5,0) -- (0,0) -- (0,2) -- (1,2) -- (1,0) -- (2,0) -- (2,2) -- (3,2) -- (3,4) -- (3.5,4) ;
    \draw[ultra thick, dotted, red] (-.5,4) -- (0,4) -- (0,6) -- (1,6) -- (1,4) -- (2,4) -- (2,6) -- (3,6) -- (3,8) -- (3.5,8) ;
    \draw[ultra thick, dotted, red] (-.5,8) -- (0,8) -- (0,10) -- (1,10) -- (1,8) -- (2,8) -- (2,10) -- (3,10) -- (3,12) -- (3.5,12) ;
    \draw[ultra thick, , dotted, red] (-.5,12) -- (0,12) -- (0,14) -- (1,14) -- (1,12) -- (2,12) -- (2,14) -- (3,14) -- (3,15.5) ;
   \draw[ultra thick, dotted, red] (3,-.5) -- (3,0) -- (3.5,0);

    \draw[ultra thick] (2,-.5) -- (2,0) -- (3,0) -- (3,2) --(3.5,2)  ;
        \draw[ultra thick] (0,-.5) -- (0,0) -- (1,0) -- (1,-.5)  ;
    \draw[ultra thick] (-.5,2) -- (0,2) -- (0,4) -- (1,4) -- (1,2) -- (2,2) -- (2,4) -- (3,4) -- (3,6) -- (3.5,6) ;
        \draw[ultra thick] (-.5,6) -- (0,6) -- (0,8) -- (1,8) -- (1,6) -- (2,6) -- (2,8) -- (3,8) -- (3,10) -- (3.5,10) ;
        \draw[ultra thick] (-.5,10) -- (0,10) -- (0,12) -- (1,12) -- (1,10) -- (2,10) -- (2,12) -- (3,12) -- (3,14) -- (3.5,14) ;
        
    \draw[ultra thick] (-.5,14) -- (0,14) -- (0,15.5) ;
        \draw[ultra thick] (1, 15.5) -- (1,14) -- (2,14) -- (2,15.5) ;

    \foreach \x in {0,2}{
    \foreach \y in {2, 3, 6, 7, 10, 11, 14, 15}{
      \node[style={circle,fill=black},scale=.6] at (\x,\y){};
}
}

    \foreach \x in {1}{
    \foreach \y in {0, 3, 4, 7, 8, 11, 12, 15}{
      \node[style={circle,fill=black},scale=.6] at (\x,\y){};
}
}

    \foreach \x in {3}{
    \foreach \y in {0, 1, 4, 5, 8, 9, 12, 13}{
      \node[style={circle,fill=black},scale=.6] at (\x,\y){};
}
}

    \foreach \x in {0,2}{
    \foreach \y in {0, ..., 3}{
      \node[style={rectangle,fill=red},scale=.6] at (\x,4*\y){};
      \node[style={rectangle,fill=red},scale=.6] at (\x,4*\y+1){};
}
}

    \foreach \x in {1}{
    \foreach \y in {0, ..., 3}{
      \node[style={rectangle,fill=red},scale=.6] at (\x,4*\y+2){};
      \node[style={rectangle,fill=red},scale=.6] at (\x,4*\y+1){};
}
}

    \foreach \x in {3}{
    \foreach \y in {0, ..., 3}{
      \node[style={rectangle,fill=red},scale=.6] at (\x,4*\y+2){};
      \node[style={rectangle,fill=red},scale=.6] at (\x,4*\y+3){};
}
}

      \node (0) at (-.5,0) [label=left:$0000$]{};
      \node (1) at (-.5,1) [label=left:$0100$]{};
      \node (2) at (-.5,2) [label=left:$0101$]{};
      \node (3) at (-.5,3) [label=left:$0001$]{};
      \node (4) at (-.5,4) [label=left:$0011$]{};
      \node (5) at (-.5,5) [label=left:$0111$]{};
      \node (6) at (-.5,6) [label=left:$0110$]{};
      \node (7) at (-.5,7) [label=left:$1110$]{};
      \node (8) at (-.5,8) [label=left:$1100$]{};
      \node (9) at (-.5,9) [label=left:$1000$]{};
      \node (10) at (-.5,10) [label=left:$1001$]{};
      \node (11) at (-.5,11) [label=left:$1101$]{};
      \node (12) at (-.5,12) [label=left:$1111$]{};
      \node (13) at (-.5,13) [label=left:$1011$]{};
      \node (14) at (-.5,14) [label=left:$1010$]{};
      \node (15) at (-.5,15) [label=left:$0010$]{};

      \node (a) at (-0,-.5) [label=below:$00$]{};
      \node (b) at (1,-.5) [label=below:$01$]{};
      \node (c) at (2,-.5) [label=below:$11$]{};
      \node (d) at (3,-.5) [label=below:$10$]{};

\end{tikzpicture}
\hspace{.5in}
\begin{tikzpicture}[scale=.5]  

    \draw[ultra thick, dotted, red] (-.5,0) -- (0,0) -- (0,2) -- (1,2) -- (1,0) -- (2,0) -- (2,2) -- (3,2) -- (3,4) -- (3.5,4) ;
    \draw[ultra thick, dotted, red] (-.5,4) -- (0,4) -- (0,6) -- (1,6) -- (1,4) -- (2,4) -- (2,6) -- (3,6) -- (3,8) -- (3.5,8) ;
    \draw[ultra thick, dotted, red] (-.5,8) -- (0,8) -- (0,10) -- (1,10) -- (1,8) -- (2,8) -- (2,10) -- (3,10) -- (3,12) -- (3.5,12) ;
    \draw[ultra thick, dotted, red] (-.5,12) -- (0,12) -- (0,14) -- (1,14) -- (1,12) -- (2,12) -- (2,14) -- (3,14) -- (3,14.5) ;
   \draw[ultra thick, dotted, red] (3,-1.5) -- (3,0) -- (3.5,0);

    \draw[ultra thick] (2,-1.5) -- (2,0) -- (3,0) -- (3,2) --(3.5,2)  ;
        \draw[ultra thick] (0,-1.5) -- (0,0) -- (1,0) -- (1,-1.5)  ;
    \draw[ultra thick] (-.5,2) -- (0,2) -- (0,4) -- (1,4) -- (1,2) -- (2,2) -- (2,4) -- (3,4) -- (3,6) -- (3.5,6) ;
        \draw[ultra thick] (-.5,6) -- (0,6) -- (0,8) -- (1,8) -- (1,6) -- (2,6) -- (2,8) -- (3,8) -- (3,10) -- (3.5,10) ;
        \draw[ultra thick] (-.5,10) -- (0,10) -- (0,12) -- (1,12) -- (1,10) -- (2,10) -- (2,12) -- (3,12) -- (3,14) -- (3.5,14) ;
        
    \draw[ultra thick] (-.5,14) -- (0,14) -- (0,14.5) ;
        \draw[ultra thick] (1, 14.5) -- (1,14) -- (2,14) -- (2,14.5) ;

    \foreach \x in {0, 2}{
    \foreach \y in {-1,2, 3, 6, 7, 10, 11, 14}{
      \node[style={circle,fill=black},scale=.6] at (\x,\y){};
}
}
    \foreach \x in {1}{
    \foreach \y in {-1,0, 3, 4, 7, 8, 11, 12}{
      \node[style={circle,fill=black},scale=.6] at (\x,\y){};
}
}

    \foreach \x in {3}{
    \foreach \y in {0, 1, 4, 5, 8, 9, 12, 13}{
      \node[style={circle,fill=black},scale=.6] at (\x,\y){};
}
}

    \foreach \x in {0,2}{
    \foreach \y in {0, ..., 3}{
      \node[style={rectangle,fill=red},scale=.6] at (\x,4*\y){};
      \node[style={rectangle,fill=red},scale=.6] at (\x,4*\y+1){};
}
}

    \foreach \x in {1}{
    \foreach \y in {0, ..., 3}{
      \node[style={rectangle,fill=red},scale=.6] at (\x,4*\y+2){};
      \node[style={rectangle,fill=red},scale=.6] at (\x,4*\y+1){};
}
}

    \foreach \x in {3}{
    \foreach \y in {0, ..., 2}{
      \node[style={rectangle,fill=red},scale=.6] at (\x,4*\y+2){};
      \node[style={rectangle,fill=red},scale=.6] at (\x,4*\y+3){};
}
}

      \node[style={rectangle,fill=red},scale=.6] at (3,14){};
      \node[style={rectangle,fill=red},scale=.6] at (3,-1){};

      \node (0) at (-.5,-1) [label=left:$0000$]{};
      \node (1) at (-.5,0) [label=left:$0001$]{};
      \node (2) at (-.5,1) [label=left:$1001$]{};
      \node (3) at (-.5,2) [label=left:$1011$]{};
      \node (4) at (-.5,3) [label=left:$0011$]{};
      \node (5) at (-.5,4) [label=left:$0010$]{};
      \node (6) at (-.5,5) [label=left:$0110$]{};
      \node (7) at (-.5,6) [label=left:$0100$]{};
      \node (8) at (-.5,7) [label=left:$1100$]{};
      \node (9) at (-.5,8) [label=left:$1101$]{};
      \node (10) at (-.5,9) [label=left:$0101$]{};
      \node (11) at (-.5,10) [label=left:$0111$]{};
      \node (12) at (-.5,11) [label=left:$1111$]{};
      \node (13) at (-.5,12) [label=left:$1110$]{};
      \node (14) at (-.5,13) [label=left:$1010$]{};
      \node (15) at (-.5,14) [label=left:$1000$]{};

      \node (a) at (-0,-1.5) [label=below:$00$]{};
      \node (b) at (1,-1.5) [label=below:$01$]{};
      \node (c) at (2,-1.5) [label=below:$11$]{};
      \node (d) at (3,-1.5) [label=below:$10$]{};

\end{tikzpicture}
\end{center}
\caption{A 2-splittable decomposition of $Q_6$ into four cycles of the same length.  $\F_1$ contains the two cycles on the left (black and dotted red), where the red square vertices and the black circular vertices are the representing sets for the cycle with the same color.  $\F_2$ contains the two cycles on the right, with the same color scheme for cycles and representing sets.}\label{642split}
\end{figure}


\begin{lem}\label{p7gen}
Suppose the graph $G$ has an  $(a,b)$-DR-splittable decomposition into $a m$ cycles of the same even length 
and the graph $G'$ has a $c$-splittable decomposition 
into $cm$ cycles of the same length
such that the representing sets in both decompositions have an even number of vertices.
Then $G \sq G'$ has a $2bc$-splittable decomposition into $2m a c$ cycles of the same length, where all representing sets have an even number of vertices.
\end{lem}

Before giving the proof, we consider some examples: Figure~\ref{642split} illustrates how Lemma~\ref{p7gen} is applied to decompose $Q_6$ into 4 cycles.  In this example, we write $Q_6 = Q_4 \sq Q_2$, where $Q_4$ has a $(2,1)$-DR-splittable decomposition into two 16-cycles
\begin{align*}
C_1 = ( &0000, 0100,0101, 0001, 0011, 0111, 0110, 1110,\\
& 1100, 1000, 1001, 1101, 1111, 1011, 1010, 0010, 0000)
\end{align*}
and 
\begin{align*}
C_2 = ( &0000,0001,1001,1011, 0011,0010,0110, 0100,\\
& 1100, 1101, 0101, 0111, 1111, 1110, 1010, 1000,0000)
\end{align*}
with representing sets 
\[S_1 = \{0000, 0101, 0011, 0110, 1100,  1001,  1111, 1010\}\]
and 
\[S_2 = \{0001,1011, 0010, 0100, 1101, 0111, 1110,  1000\},\]
respectively. We know $Q_2$ has a $1$-splittable decomposition into  one 4-cycle $(00,01,11,10,00)$.   So $a =2$, $b=1$, $c=1$, and $m=1$, giving $2bc = 2$ and $2 m a c = 4$.  Thus the result is a 2-splittable decomposition into 4 cycles. The two cycles in each subdivided torus split the vertices of $Q_6$. 
Note that the vector corresponding to any vertex in $Q_6$ in the figure can be found by concatenating the vector to its left and the vector below.

Figure~\ref{4lsplit} illustrates how Lemma~\ref{p7gen} is applied to decompose  $Q_6$ into 8 cycles.  Again, we write $Q_6 = Q_4 \sq Q_2$, where $Q_4$ has a $(4,2)$-DR-splittable decomposition into four 8-cycles
\begin{align*}
C_1 &= (0000,0100,0101,1101,1111,1011, 1010, 0010, 0000),\\
C_2 &= (1100, 1000, 1001, 0001, 0011, 0111, 0110, 1110, 1100),\\ 
C_3 &= (0100, 1100, 1101, 1001, 1011, 0011, 0010, 0110, 0100), \text{ and }\\ 
C_4 &= (1000, 0000, 0001, 0101, 0111, 1111, 1110, 1010, 1000),
\end{align*}
with representing sets 
\begin{align*}
S_1 &= \{0000, 0101, 1111, 1010 \},\\
S_2 &= \{1100,  1001, 0011,  0110 \},\\ 
S_3 &= \{0100, 1101,  1011,  0010 \}, \text{ and }\\
S_4 &= \{1000, 0001, 0111,  1110 \},
\end{align*}
respectively. In this decomposition we take $\F_1 = \{C_1, C_2, C_3, C_4\}$, with $\F_{1,1} = \{C_1, C_2\}$ and $\F_{1,2} = \{C_3, C_4\}$, i.e. $V(Q_4)$ is partitioned by $S_1 \cup S_2 \cup S_3 \cup S_4$, and also by $V(C_1) \cup V(C_2)$ and $V(C_3) \cup V(C_4)$. We know $Q_2$ has a 1-splittable decomposition into into one 4-cycle $(00,01,11,10,00)$.    So $a =4$, $b=2$, $c=1$, and $m=1$, giving $2bc = 4$, and $2 m a c = 8$.  Thus the result is a 4-splittable decomposition into 8 cycles. 

\begin{proof}

Let $C_1, \ldots, C_{a m}$ and $C_1', \ldots, C_{c m}'$ be the cycles decomposing $G$ and $G'$, respectively, with representing sets $S_1, \ldots, S_{a m}$  and  $S_1', \ldots ,S_{c m}'$. Let $\F_1, \ldots, \F_m$ be splitting sets, with splitting subsets $\F_{i,1}, \ldots, \F_{i,a/b}$ for $1 \le i \le m$, of the $(a,b)$-splittable decomposition of $G$, and $\F'_1, \ldots, \F'_m$ be the splitting sets for the $c$-splittable decomposition of $G'$.
That is, for $i=1, \ldots, m$, $\F_i$ consists of  $a$  cycles $C_s$, and can be partitioned into subsets $\F_{i,1}, \ldots, \F_{i,a/b}$ where the $b$ cycles in each $\F_{i,j}$ are vertex disjoint and span $V(G)$. Similarly, $\F'_i$ consists of  $c$  cycles $C'_t$, and $\F'_i$  forms a splittable decomposition of a spanning subgraph of $G'$, $i=1, \ldots, m$. Then

\begin{align*}
G \sq G' &=  \bigcup_{i=1}^{m} \bigcup_{C\in \F_i} C~ ~ \sq ~~ \bigcup_{i=1}^{m} \bigcup_{C'\in \F'_i}  C'& \\
&=  \bigcup_{i=1}^{m}  \left( \bigcup_{C\in \F_i }C~\sq ~\bigcup_{C'\in \F'_i} C' \right)&  \text{by Proposition~\ref{spanning}} \\
&=  \bigcup_{i=1}^{m}      \bigcup_{C_s\in \F_i} \bigcup_{C'_t\in \F'_i}  (C_s, S_s) \anc (C_t', S_t')
& \text{by Proposition~\ref{decomp}}. \\
\end{align*}

Each $(C_s, S_s) \anc (C_t', S_t')$ is a subdivided torus, denote it by $T_{s,t}$.  Recall that these tori  are pairwise edge-disjoint (see Proposition~\ref{spanning}) and the unions of anchored products are pairwise edge-disjoint (see Proposition~\ref{decomp}).  Since each $|S_s|$ and $|S_t'|$ is even, $T_{s,t}$ allows the 2-wiggle decomposition, and decomposes into two cycles, $C_{s,t}$ and $C'_{s,t}$. Since each $(C_s,S_s)$ is distance regular, by Proposition~\ref{2even}, $C_{s,t}$ and $C'_{s,t}$ have same length. This gives a decomposition of $G \sq G'$ into $2\cdot m\cdot a \cdot c$ cycles.
Since each $S_s$ has the same cardinality, and each $S_t'$ has the same cardinality,  all tori  $T_{s,t}$ have the same number of edges and thus all the resulting cycles of the decomposition have the same length.  

We need to argue that the resulting cycle decomposition is $2bc$-splittable, i.e., the cycles can be grouped into splitting sets of size $2bc$ each, where each cycle has a representing set of the same even cardinality, and the representing sets for a given splitting set partition $V(G \sq G')$.
For $i=1, \ldots, m$, $j=1, \ldots, a/b$, let the splitting set ${\cal H}_{i,j} = \{C_{s, t}, C'_{s, t}:  C_s \in \F_{i,j}, C_t'\in \F_i'\}$. Note that each ${\cal H}_{i,j}$ contains $2bc$ cycles, and each cycle in the decomposition is in exactly one such set. It remains to assign representing sets of even cardinality to each cycle in ${\cal H}_{i,j}$ so that they partition $V(G \sq G')$.

Fix $i$ and $j$.  Given $C_s \in \F_{i,j}$ and $C_t'\in \F_i'$ we will split the vertices in each $V(C_s) \times S'_t$ into two sets $S_{s,t}$ and $S_{s,t}'$ to form representing sets for  $C_{s, t}$ and $C'_{s, t}$.  First we verify that this will partition the vertices in $V(G\sq G')$.
Since the sets $\{V(C_s) : C_s \in \F_{i,j}\}$ partition $V(G)$, for a given $t$, the sets $\{V(C_s) \times S'_t : ~ C_s \in \F_{i,j} \}$ partition $V(G) \times S_t'$. Since the sets $\{S_t' : ~ C_t'\in \F_i'\}$ partition $V(G')$, the set $\{V(C_s) \times S'_t : ~ C_s \in \F_{i,j}, C_t' \in \F_i' \}$ partitions $V(G) \times V(G') = V(G\sq G')$. 
  
Since $(C_s,S_s)$ is distance regular, Proposition~\ref{subsplit} assures that we can find $S_{s,t} \subseteq V(C_{s,t})$ and $S_{s,t}' \subseteq V(C'_{s,t})$ where these sets have the same cardinality and partition $V(C_s)\times S_t'$. Further, since every $C_s$ is of the same even length and every $S_t'$ has the same even cardinality, for every $C_s \in \F_{i,j}, C_t' \in \F_i' $, the set $V(C_s) \times S'_t$ contains the same number of vertices, and this number is a multiple of four. This implies the number of vertices in $S_{s,t}$ and $S_{s,t}'$ is even.
\end{proof}

\begin{figure}

\begin{tikzpicture}[scale=.5]  


    \draw[ultra thick, red, dotted] (-.5,0) -- (0,0) -- (0,2) -- (1,2) -- (1,0) -- (2,0) -- (2,2) -- (3,2) -- (3,4) -- (3.5,4) ;
    \draw[ultra thick, red, dotted] (-.5,4) -- (0,4) -- (0,6) -- (1,6) -- (1,4) -- (2,4) -- (2,6) -- (3,6) -- (3,7.5);
   \draw[ultra thick, red, dotted] (3,-.5) -- (3,0) -- (3.5,0);
   
       \draw[ultra thick] (0,-.5) -- (0,0) -- (1,0) -- (1,-.5) ;
    \draw[ultra thick] (-.5,2) -- (0,2) -- (0,4) -- (1,4) -- (1,2) -- (2,2) -- (2,4) -- (3,4) -- (3,6) -- (3.5,6);
   \draw[ultra thick] (2,-.5) -- (2,0) -- (3,0) -- (3,2) -- (3.5,2);
    \draw[ultra thick] (-.5,6) -- (0,6) -- (0,7.5) ;
        \draw[ultra thick] (1,7.5) -- (1,6) -- (2,6) -- (2,7.5) ;

    \foreach \x in {0,2}{
    \foreach \y in {2, 3, 6, 7}{
      \node[style={circle,fill=black},scale=.6] at (\x,\y){};
}
}

    \foreach \x in {1}{
    \foreach \y in {0,3, 4, 7}{
      \node[style={circle,fill=black},scale=.6] at (\x,\y){};
}
}

    \foreach \x in {3}{
    \foreach \y in {0, 1, 4, 5,}{
      \node[style={circle,fill=black},scale=.6] at (\x,\y){};
}
}

    \foreach \x in {0, 2}{
    \foreach \y in {0, 1}{
      \node[style={rectangle,fill=red},scale=.6] at (\x,4*\y){};
      \node[style={rectangle,fill=red},scale=.6] at (\x,4*\y+1){};
}
}

    \foreach \x in {1}{
    \foreach \y in {0, 1}{
      \node[style={rectangle,fill=red},scale=.6] at (\x,4*\y+2){};
      \node[style={rectangle,fill=red},scale=.6] at (\x,4*\y+1){};
}
}

    \foreach \x in {3}{
    \foreach \y in {0, 1}{
      \node[style={rectangle,fill=red},scale=.6] at (\x,4*\y+2){};
      \node[style={rectangle,fill=red},scale=.6] at (\x,4*\y+3){};
}
}

      \node (0) at (-.5,0) [label=left:$0000$]{};
      \node (1) at (-.5,1) [label=left:$0100$]{};
      \node (2) at (-.5,2) [label=left:$0101$]{};
      \node (3) at (-.5,3) [label=left:$1101$]{};
      \node (4) at (-.5,4) [label=left:$1111$]{};
      \node (5) at (-.5,5) [label=left:$1011$]{};
      \node (6) at (-.5,6) [label=left:$1010$]{};
      \node (7) at (-.5,7) [label=left:$0010$]{};

      \node (4) at (0,-.5) [label=below:$00$]{};
      \node (5) at (1, -.5) [label=below:$01$]{};
      \node (6) at (2, -.5) [label=below:$11$]{};
      \node (7) at (3,-.5) [label=below:$10$]{};

\end{tikzpicture}
\hspace{.2in}
\begin{tikzpicture}[scale=.5]  

    \draw[ultra thick, red, dotted] (-.5,0) -- (0,0) -- (0,2) -- (1,2) -- (1,0) -- (2,0) -- (2,2) -- (3,2) -- (3,4) -- (3.5,4) ;
    \draw[ultra thick, red, dotted] (-.5,4) -- (0,4) -- (0,6) -- (1,6) -- (1,4) -- (2,4) -- (2,6) -- (3,6) -- (3,7.5);
   \draw[ultra thick, red, dotted] (3,-.5) -- (3,0) -- (3.5,0);
   
       \draw[ultra thick] (0,-.5) -- (0,0) -- (1,0) -- (1,-.5) ;
    \draw[ultra thick] (-.5,2) -- (0,2) -- (0,4) -- (1,4) -- (1,2) -- (2,2) -- (2,4) -- (3,4) -- (3,6) -- (3.5,6);
   \draw[ultra thick] (2,-.5) -- (2,0) -- (3,0) -- (3,2) -- (3.5,2);
    \draw[ultra thick] (-.5,6) -- (0,6) -- (0,7.5) ;
        \draw[ultra thick] (1,7.5) -- (1,6) -- (2,6) -- (2,7.5) ;

    \foreach \x in {0,2}{
    \foreach \y in {2, 3, 6, 7}{
      \node[style={circle,fill=black},scale=.6] at (\x,\y){};
}
}

    \foreach \x in {1}{
    \foreach \y in {0,3, 4, 7}{
      \node[style={circle,fill=black},scale=.6] at (\x,\y){};
}
}

    \foreach \x in {3}{
    \foreach \y in {0, 1, 4, 5,}{
      \node[style={circle,fill=black},scale=.6] at (\x,\y){};
}
}

    \foreach \x in {0, 2}{
    \foreach \y in {0, 1}{
      \node[style={rectangle,fill=red},scale=.6] at (\x,4*\y){};
      \node[style={rectangle,fill=red},scale=.6] at (\x,4*\y+1){};
}
}

    \foreach \x in {1}{
    \foreach \y in {0, 1}{
      \node[style={rectangle,fill=red},scale=.6] at (\x,4*\y+2){};
      \node[style={rectangle,fill=red},scale=.6] at (\x,4*\y+1){};
}
}

    \foreach \x in {3}{
    \foreach \y in {0, 1}{
      \node[style={rectangle,fill=red},scale=.6] at (\x,4*\y+2){};
      \node[style={rectangle,fill=red},scale=.6] at (\x,4*\y+3){};
}
}

      \node (0) at (-.5,0) [label=left:$1100$]{};
      \node (1) at (-.5,1) [label=left:$1000$]{};
      \node (2) at (-.5,2) [label=left:$1001$]{};
      \node (3) at (-.5,3) [label=left:$0001$]{};
      \node (4) at (-.5,4) [label=left:$0011$]{};
      \node (5) at (-.5,5) [label=left:$0111$]{};
      \node (6) at (-.5,6) [label=left:$0110$]{};
      \node (7) at (-.5,7) [label=left:$1110$]{};

      \node (4) at (0,-.5) [label=below:$00$]{};
      \node (5) at (1, -.5) [label=below:$01$]{};
      \node (6) at (2, -.5) [label=below:$11$]{};
      \node (7) at (3,-.5) [label=below:$10$]{};

\end{tikzpicture}
\hfill
\begin{tikzpicture}[scale=.5]  

    \draw[ultra thick, red, dotted] (-.5,0) -- (0,0) -- (0,2) -- (1,2) -- (1,0) -- (2,0) -- (2,2) -- (3,2) -- (3,4) -- (3.5,4) ;
    \draw[ultra thick, red, dotted] (-.5,4) -- (0,4) -- (0,6) -- (1,6) -- (1,4) -- (2,4) -- (2,6) -- (3,6) -- (3,7.5);
   \draw[ultra thick, red, dotted] (3,-.5) -- (3,0) -- (3.5,0);
   
       \draw[ultra thick] (0,-.5) -- (0,0) -- (1,0) -- (1,-.5) ;
    \draw[ultra thick] (-.5,2) -- (0,2) -- (0,4) -- (1,4) -- (1,2) -- (2,2) -- (2,4) -- (3,4) -- (3,6) -- (3.5,6);
   \draw[ultra thick] (2,-.5) -- (2,0) -- (3,0) -- (3,2) -- (3.5,2);
    \draw[ultra thick] (-.5,6) -- (0,6) -- (0,7.5) ;
        \draw[ultra thick] (1,7.5) -- (1,6) -- (2,6) -- (2,7.5) ;

    \foreach \x in {0,2}{
    \foreach \y in {2, 3, 6, 7}{
      \node[style={circle,fill=black},scale=.6] at (\x,\y){};
}
}

    \foreach \x in {1}{
    \foreach \y in {0,3, 4, 7}{
      \node[style={circle,fill=black},scale=.6] at (\x,\y){};
}
}

    \foreach \x in {3}{
    \foreach \y in {0, 1, 4, 5,}{
      \node[style={circle,fill=black},scale=.6] at (\x,\y){};
}
}

    \foreach \x in {0, 2}{
    \foreach \y in {0, 1}{
      \node[style={rectangle,fill=red},scale=.6] at (\x,4*\y){};
      \node[style={rectangle,fill=red},scale=.6] at (\x,4*\y+1){};
}
}

    \foreach \x in {1}{
    \foreach \y in {0, 1}{
      \node[style={rectangle,fill=red},scale=.6] at (\x,4*\y+2){};
      \node[style={rectangle,fill=red},scale=.6] at (\x,4*\y+1){};
}
}

    \foreach \x in {3}{
    \foreach \y in {0, 1}{
      \node[style={rectangle,fill=red},scale=.6] at (\x,4*\y+2){};
      \node[style={rectangle,fill=red},scale=.6] at (\x,4*\y+3){};
}
}

      \node (0) at (-.5,0) [label=left:$0100$]{};
      \node (1) at (-.5,1) [label=left:$1100$]{};
      \node (2) at (-.5,2) [label=left:$1101$]{};
      \node (3) at (-.5,3) [label=left:$1001$]{};
      \node (4) at (-.5,4) [label=left:$1011$]{};
      \node (5) at (-.5,5) [label=left:$0011$]{};
      \node (6) at (-.5,6) [label=left:$0010$]{};
      \node (7) at (-.5,7) [label=left:$0110$]{};

      \node (4) at (0,-.5) [label=below:$00$]{};
      \node (5) at (1, -.5) [label=below:$01$]{};
      \node (6) at (2, -.5) [label=below:$11$]{};
      \node (7) at (3,-.5) [label=below:$10$]{};

\end{tikzpicture}
\hspace{.2in}
\begin{tikzpicture}[scale=.5]  

    \draw[ultra thick, red, dotted] (-.5,0) -- (0,0) -- (0,2) -- (1,2) -- (1,0) -- (2,0) -- (2,2) -- (3,2) -- (3,4) -- (3.5,4) ;
    \draw[ultra thick, red, dotted] (-.5,4) -- (0,4) -- (0,6) -- (1,6) -- (1,4) -- (2,4) -- (2,6) -- (3,6) -- (3,7.5);
   \draw[ultra thick, red, dotted] (3,-.5) -- (3,0) -- (3.5,0);
   
       \draw[ultra thick] (0,-.5) -- (0,0) -- (1,0) -- (1,-.5) ;
    \draw[ultra thick] (-.5,2) -- (0,2) -- (0,4) -- (1,4) -- (1,2) -- (2,2) -- (2,4) -- (3,4) -- (3,6) -- (3.5,6);
   \draw[ultra thick] (2,-.5) -- (2,0) -- (3,0) -- (3,2) -- (3.5,2);
    \draw[ultra thick] (-.5,6) -- (0,6) -- (0,7.5) ;
        \draw[ultra thick] (1,7.5) -- (1,6) -- (2,6) -- (2,7.5) ;

    \foreach \x in {0,2}{
    \foreach \y in {2, 3, 6, 7}{
      \node[style={circle,fill=black},scale=.6] at (\x,\y){};
}
}

    \foreach \x in {1}{
    \foreach \y in {0,3, 4, 7}{
      \node[style={circle,fill=black},scale=.6] at (\x,\y){};
}
}

    \foreach \x in {3}{
    \foreach \y in {0, 1, 4, 5,}{
      \node[style={circle,fill=black},scale=.6] at (\x,\y){};
}
}

    \foreach \x in {0, 2}{
    \foreach \y in {0, 1}{
      \node[style={rectangle,fill=red},scale=.6] at (\x,4*\y){};
      \node[style={rectangle,fill=red},scale=.6] at (\x,4*\y+1){};
}
}

    \foreach \x in {1}{
    \foreach \y in {0, 1}{
      \node[style={rectangle,fill=red},scale=.6] at (\x,4*\y+2){};
      \node[style={rectangle,fill=red},scale=.6] at (\x,4*\y+1){};
}
}

    \foreach \x in {3}{
    \foreach \y in {0, 1}{
      \node[style={rectangle,fill=red},scale=.6] at (\x,4*\y+2){};
      \node[style={rectangle,fill=red},scale=.6] at (\x,4*\y+3){};
}
}

      \node (0) at (-.5,0) [label=left:$1000$]{};
      \node (1) at (-.5,1) [label=left:$0000$]{};
      \node (2) at (-.5,2) [label=left:$0001$]{};
      \node (3) at (-.5,3) [label=left:$0101$]{};
      \node (4) at (-.5,4) [label=left:$0111$]{};
      \node (5) at (-.5,5) [label=left:$1111$]{};
      \node (6) at (-.5,6) [label=left:$1110$]{};
      \node (7) at (-.5,7) [label=left:$1010$]{};

      \node (4) at (0,-.5) [label=below:$00$]{};
      \node (5) at (1, -.5) [label=below:$01$]{};
      \node (6) at (2, -.5) [label=below:$11$]{};
      \node (7) at (3,-.5) [label=below:$10$]{};

\end{tikzpicture}

\caption{A 4-splittable decomposition of $Q_6$ into eight cycles of the same length.  $\F_1$ contains the four cycles on the left two subdivided tori, where the red square vertices and the black circular vertices are the representing sets for the red dotted and black cycles, respectively.  $\F_2$ contains the four cycles on the right two subdivided tori, with the same color scheme for cycles and representing sets.}\label{4lsplit}
\end{figure}
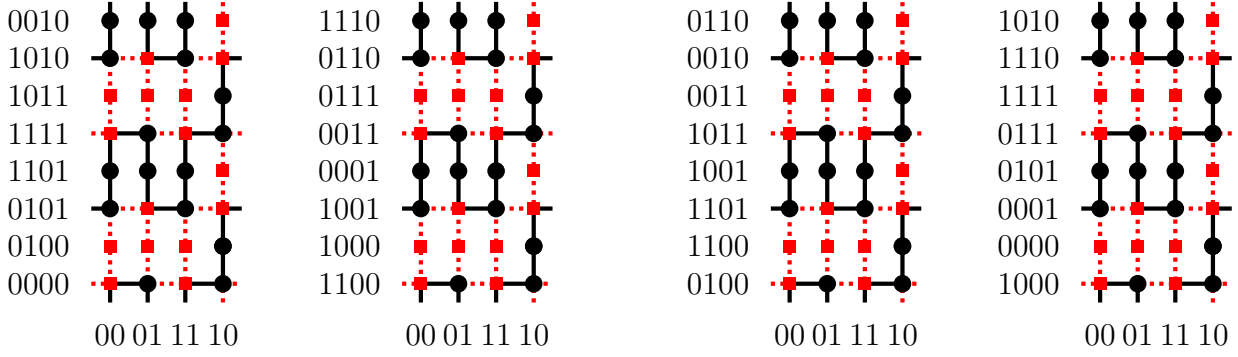

\subsection{Decomposition of products without increasing cycle length}

In this subsection, we prove, under two different splittability conditions, two propositions which imply that if $G$ has a decomposition into cycles of a given length, then $G \sq G$ has a decomposition into cycles of the same length.

\begin{prop}\label{dec}
If $G$ has an $(a,b)$-DR-splittable cycle decomposition into cycles of length $\ell$, then $G \sq G$ has an $(a|V(G)|, b|V(G)|)$-DR-splittable decomposition into cycles of length $\ell$.
\end{prop}

\begin{proof}
Let $\F_1, \ldots, \F_{m}$ be the splitting sets of the  $(a,b)$-DR-splittable cycle decomposition of $G$, with splitting subsets $\F_{i,1}, \ldots, \F_{i,a/b}$ for $1 \le i \le m$. Recall the definition of vertical and horizontal graphs and cycles given in Section~\ref{splittable}. The product $G \sq G$ can be decomposed into $2|V(G)|$ edge-disjoint copies of $G$:  $|V(G)|$ horizontal copies induced by $\{(u,v): v \in V(G)\}$ for a fixed $u \in V(G)$, and $|V(G)|$ vertical copies induced by $\{(u,v): u \in V(G)\}$ for a fixed $v \in V(G)$. Copy the cycle decomposition of $G$ into each of these copies to obtain a cycle decomposition of $G \sq G$. For $1 \le i \le m$, let $\F_i'$ consist of all images of the cycles in the splitting set $\F_i$ in the horizontal cycles. Then $\F_i'$ contains $a|V(G)|$ cycles.  For representing sets, assign to each cycle the image of its representing set from the decomposition of $G$.  Since the representing sets in $\F_i$ partition $V(G)$, the representing sets in $\F_i'$ partition $V(G \sq G)$, and are still distance regular. For $1 \le i \le m$, $i \le j \le a/b$, let the splitting subset $\F_{i,j}'$ contain the image of all cycles from $\F_{i,j}$ in the horizontal copies of $G$. Note that each $\F_{i,j}'$ contains $b|V(G)|$ cycles and the vertices in these cycles partition $V(G \sq G)$.  Doing the same thing with the vertical copies of $G$ creates more splitting sets $\F_i''$, with splitting subsets $\F_{i,j}''$, and together all of the splitting sets $\F_i'$ and $F_i''$ with splitting subsets $\F_{i,j}'$ and $\F_{i,j}''$ give an $(a|V(G)|, b|V(G)|)$-DR-splittable cycle decomposition of $G \sq G$ into cycles of length $\ell$.
\end{proof}

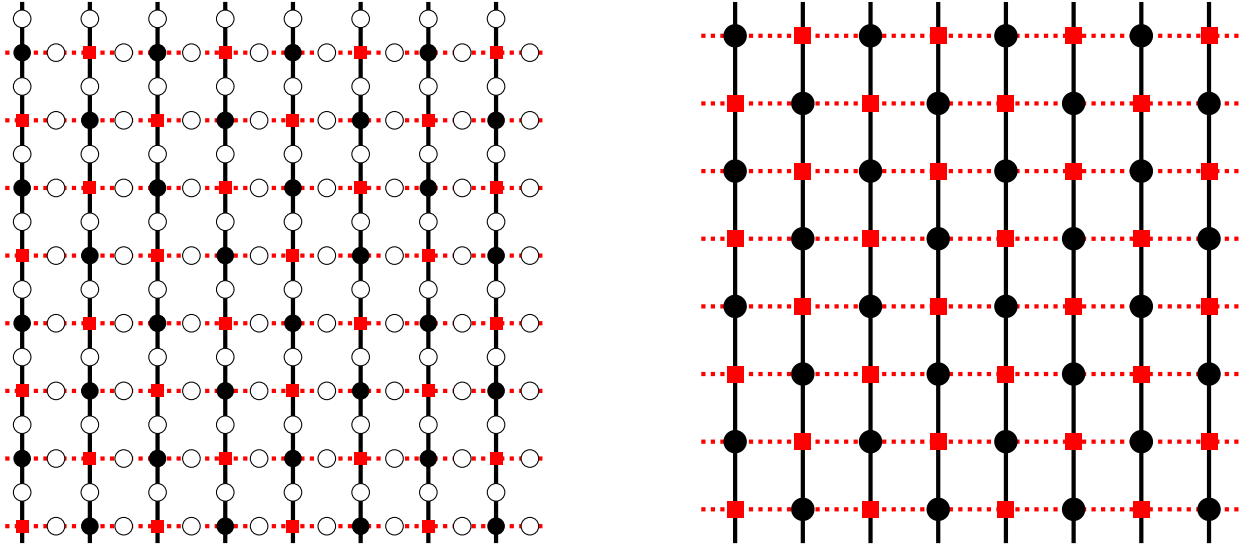
\begin{figure}
\begin{center}
\begin{tikzpicture}[scale=.45]  

    \foreach \x in {0,2, 4, 6, 8, 10, 12, 14}{
    \draw[ultra thick] (\x,-.5) -- (\x,15.5);
}
    \foreach \y in {0,2, 4, 6, 8, 10, 12, 14}{
    \draw[ultra thick, red, dotted] (-.5,\y) -- (15.5,\y);
}

    \foreach \x in {0, 4, 8, 12}{
    \foreach \y in {0, 4, 8, 12}{
      \node[style={rectangle,fill=red},scale=.6] at (\x,\y){};
}
}

    \foreach \x in {2, 6, 10, 14}{
    \foreach \y in {2, 6, 10, 14}{
      \node[style={rectangle,fill=red},scale=.6] at (\x,\y){};
}
}

    \foreach \x in {0, 4, 8, 12}{
    \foreach \y in {2, 6, 10, 14}{
      \node[style={circle,fill=black},scale=.6] at (\x,\y){};
}
}

    \foreach \x in {2, 6, 10, 14}{
    \foreach \y in {0, 4, 8, 12}{
      \node[style={circle,fill=black},scale=.6] at (\x,\y){};
}
}

    \foreach \x in {0, 2,4, 6, 8, 10, 12, 14}{
    \foreach \y in {1, 3, 5, 7, 9, 11, 13, 15}{
      \node[style={circle,draw, fill=white},scale=.6] at (\x,\y){};
}
}
    \foreach \y in {0, 2,4, 6, 8, 10,  12, 14}{
    \foreach \x in {1, 3, 5, 7, 9, 11, 13, 15}{
      \node[style={circle,draw, fill=white},scale=.6] at (\x,\y){};
}
}

\end{tikzpicture}
\hfill
\begin{tikzpicture}[scale=.9]  

    \foreach \x in {0,...,7}{
    \draw[ultra thick] (\x,-.5) -- (\x,7.5);
}
    \foreach \y in {0,...,7}{
    \draw[ultra thick, red, dotted] (-.5,\y) -- (7.5,\y);
}

    \foreach \x in {0, 2, 4, 6}{
    \foreach \y in {0, 2, 4, 6}{
      \node[style={rectangle,fill=red},scale=.8] at (\x,\y){};
}
}

    \foreach \x in {0, 2, 4, 6}{
    \foreach \y in {1, 3, 5, 7}{
      \node[style={circle,fill=black},scale=.8] at (\x,\y){};
}
}

    \foreach \x in {1, 3, 5, 7}{
    \foreach \y in {0, 2, 4, 6}{
      \node[style={circle,fill=black},scale=.8] at (\x,\y){};
}
}

    \foreach \x in {1, 3, 5, 7}{
    \foreach \y in {1, 3, 5, 7}{
      \node[style={rectangle,fill=red},scale=.8] at (\x,\y){};
}
}
\end{tikzpicture}

\end{center}
\caption{An illustration of the assignment of the degree four vertices in the subdivided tori to representing sets in Proposition~\ref{GboxG}. Left: The red square vertices are in the representing set for the red dotted horizontal cycles, and the black circular vertices are in the representing sets for the vertical black cycles.  Right: The corresponding cycles and vertices in the underlying torus.}\label{2lsplit}
\end{figure}

\begin{prop}\label{GboxG}
Let $G$ be a graph with an $(a,b)$-DR-splittable decomposition into cycles of length $\ell$, where $|V(G)|/a$ is even and greater than two.  Then $G \sq G$ has a  $(2a|V(G)|,b|V(G)|)$-DR-splittable decompositon into cycles of length $\ell$, where each representing set has cardinality at least two.
\end{prop}

\begin{proof}
Let  $\F_1, \ldots, \F_{m}$ be the splitting sets of of the $(a,b)$-DR-splittable decomposition of $G$. 
First we shall only use the property that this decomposition is $a$-DR-splittable.
Let $G_i$ denote the union of graphs in $\F_i$, and note that each $G_i$ is a spanning subgraph of $G$. 
By Proposition~\ref{spanning},  $G \sq G$ can be decomposed as
$\displaystyle
G \sq G = \bigcup_{i=1}^m G_i \sq G_i.
$

We now focus on decomposing each of the products $G_i \sq G_i$ in the union. Let $\F_i = \{C_1, \ldots, C_a\}$, with representative sets $S_1, \ldots, S_a$. 
By Proposition~\ref{decomp}, $G_i \sq G_i$ can be decomposed as
$\displaystyle{
G_i \sq G_i = \bigcup_{C_s \in \F_i} \bigcup_{C_t\in \F_i}  (C_s, S_s) \anc (C_t, S_t).
}$

Since $|S_s| = |V(G)|/a$,  for $s=1, \ldots, a$, each of the $a^2$ subdivided tori $(C_s,S_{s}) \anc (C_{t},S_t)$ has $|V(G)|/a$ vertical cycles and $|V(G)|/a$ horizontal cycles, each of length $\ell$.  We choose  the set $\F_i'$ of all of the horizontal and vertical cycles in all $a^2$ subdivided tori as our decomposition of $G_i\sq G_i$, and thus $|{\F}_i' |= 2(|V(G)|/a)a^2=2a|V(G)|$, $i = 1, \ldots,  m$. We now assign representing sets as illustrated in Figure~\ref{2lsplit} (left):
Each vertex in $V(G_i \sq G_i)=V(G \sq G)$ appears once as a degree $4$ vertex in exactly one of the subdivided tori $(C_s,S_{s}) \anc (C_{t},S_t)$.  Thus to assign each vertex in $G \sq G$ to exactly one representing set, we only assign to a given cycle degree four vertices from its subdivided torus, and we can instead focus on the underlying torus, as shown in Figure~\ref{2lsplit} (right).  In the underlying torus, properly two-color the vertices red and black, assigning the red vertices to be the representing sets of the horizontal cycle that they are on, and assigning the black vertices to be the representing sets for the vertical cycles they are on.  Since there is only one proper two-coloring, and this coloring alternates red and black on every horizontal and vertical cycle, each representing set is the same cardinality.  Further, since every other vertex is chosen, in the subdivided torus, these representing sets split the cycles from ${\cal F}_i'$ into paths twice as long as the corresponding paths on cycles in ${\cal F}_i$ with the original representing sets. This shows that the resulting decomposition with splitting sets $\F_1', \ldots, \F_m'$ is a $2a|V(G)|$-DR-splittable decomposition of $G\sq G$.
Note that we need $|V(G)|/a>2$ so that every cycle in ${\cal F}_i'$ has at least $2$ vertices in its representing set (Recall this is required in the definition of representing set).

Now that we have a $2a|V(G)|$-DR-splittable decomposition of $G\sq G$ with splitting sets $\F_1', \ldots, \F_m'$, we show that it is also a $(2a|V(G)|,b|V(G)|)$-DR-splittable decompositon. Since  $\F_1, \ldots, \F_m$ are splitting sets of an $(a,b)$-DR-splittable decomposition, each family $\F_i$ can be partitioned into splitting subsets $\F_{i,j}$,  each consisting of $b=|V(G)|/\ell$ cycles in $\F_i$ that are pairwise vertex disjoint and span $V(G)$, $j=1, \ldots, a/b$.

For $1 \le j \le a/b$, let $\F(V)_{i,j}'$ be all of the vertical cycles in the subdivided tori 
\[
\bigcup_{C_s\in \F_{i,j}} \bigcup_{C_t\in \F_i}  (C_s, S_s) \anc (C_t, S_t)
\]
and  let $\F(H)_{i,j}'$ be all of the horizontal cycles in the subdivided tori 
\[
\bigcup_{C_s\in \F_i} \bigcup_{C_t\in \F_{i,j}}  (C_s, S_s) \anc (C_t, S_t).
\]

For all $i$ and $j$, $\F(V)_{i,j}'$ contains a vertical copy of every cycle in $\F_{i,j}$ for every vertex in $G$. Thus it contains  $b|V(G)|$ cycles, and these cycles partition the vertices of $G \sq G$. Similarly, $\F(H)_{i,j}'$ contains a horizontal copy of every cycle in $\F_{i,j}$ for every vertex in $G$. Thus it contains  $b|V(G)|$ cycles, and these cycles partition the vertices of $G \sq G$. Finally, the union of all such sets is $\F_i'$, so the $\F(V)_{i,j}'$ and $\F(H)_{i,j}'$ are the required splitting subsets.
\end{proof}


\section{Proofs of the main results}\label{importantproofs}

First we shall prove a result about hypercube decompositions into cycles whose lengths are powers of $2$. 

\begin{lem}\label{m1m2GENERAL}
Let $x \ge 1$ be odd. For integers $n\ge 1$ and $\ell \ge 2$ where $2x \le \ell \le x2^n$, $Q_{x2^n}$ has a $(2^m, 2^{x2^n-\ell})$-DR-splittable decomposition into cycles of length $2^\ell$ for each $m$, 
\[x2^n-\ell \le m \le \min\{x2^n-1, x2^n-1+n-\ell \}.\]
\end{lem}

\begin{proof}

Let $x$ be an odd  positive integer. We have to prove the statement of the lemma for pairs $(\ell,n)$ in the allowed range.  These pairs are pictured as dots in Figure~\ref{schematic}, which contains a visualization of the order in which the cases are proved in the case $x=1$.
First we shall prove a claim that the lemma is true for pairs $(\ell,n)$ when $x2^{n-1} < \ell \le x2^n$.   These are the cases pictured as empty dots in Figure~\ref{schematic}. \\

{\bf Claim.}  For any $n\geq 1$ the following holds: if $\ell \geq 2x$   and    $x2^{n-1} < \ell \le x2^n$, then   $Q_{x2^n}$ has a  $(2^m, 2^{x2^n-\ell})$-DR-splittable decomposition into cycles of length $2^\ell$ for any $m$ such that   $x2^n-\ell \le m \le \min\{x2^n-1, x2^n-1+n -\ell\}$. \\

We shall prove  the claim by induction on $n$.   Note that $n < \ell$ for all cases considered in the claim, so  $\min\{x2^n-1, x2^n-1 +n -\ell\}= x2^n-1+n-\ell$. 

Base case $n=1$. If $n=1$ then we must have $\ell=2x$. Note that $x2^1 - \ell = 2x-2x=0$, and $x2^1-1 + 1- \ell  = 2x - 1 + 1 - 2x = 0$, so we seek a $(2^0,2^0)=(1,1)$-DR-splittable decomposition of $Q_{x2^1}=Q_{2x}$.  By the result of Aubert and Schneider~\cite{aubert} $Q_{2x}$ has a Hamiltonian decomposition into cycles of length $2^{2x}$, which is a $(1,1)$-DR-splittable decomposition of $Q_{2x}$. 

Assume the statement is true for some $n$, and fix $\ell$ such that $\ell \geq 2x$ and  $x2^n < \ell \le x2^{n+1}$. By the inductive hypothesis, using the case $\ell = x2^n$, $Q_{x2^n}$ has an $(a,b)$-DR-splittable cycle decomposition for $b=1$ and $a=2^{m'}$ for all $0 \le m' \le n-1$.
Since $b=1$, all cycles in this decomposition are Hamiltonian, with length $2^{x2^n}$.  For any $m'$ where $0 \le m' \le n-1$, suppose the splitting sets of the cycles in the $(2^{m'},1)$-DR-splittable decomposition of $Q_{x2^n}$ are $\F_1, \ldots, \F_{x2^{n-1-m'}}$.  Since all cycles are Hamiltonian, the splitting subsets $\F_{i,j}$ contain one cycle each. Then by Proposition~\ref{spanning},
$$Q_{x2^{n+1}} = Q_{x2^{n}} \sq Q_{x2^{n}} =   \bigcup_{i=1}^{x2^{n-1-m'}} \bigcup_{C\in \F_i } C \sq C.   $$ 
This gives a decomposition of $Q_{x2^{n+1}}$ into $x2^{n-1}$ tori $C \sq C$ where each cycle $C$ has length $2^{x2^n}$.  Thus each torus is a spanning subgraph of $Q_{x2^{n+1}}$, and has $2\cdot 2^{x2^n} \cdot 2^{x2^n} = 2^{x2^{n+1}+1}$ edges.  Let
$k = 2^{x2^{n+1}- \ell +1}$, and note that since $x2^n < \ell \le x2^{n+1}$,  $k$ could take any value of $2^{k'}$ where $1 \le k' \le x2^n$. Since $k$ is even and divides $2^{x2^n}$, each torus allows the $k$-wiggle decomposition, which results in each torus being decomposed into $k$ cycles, each with length $|E(C \sq C)|/k = 2^{x2^{n+1}+1}/2^{x2^{n+1}- \ell +1} = 2^\ell$.

Now we show the decomposition produced by applying the $k$-wiggle decomposition to each torus is $(2^m,2^{x2^{n+1}- \ell})$-DR-splittable for all values of $m$ where $x2^{n+1}- \ell \le m \le x2^{n+1}-1+(n+1)-\ell$. 

First we show it is  $2^m$-DR-splittable for all values of $m$ where $x2^{n+1}- \ell +1 \le m \le x2^{n+1}-1+(n+1)-\ell$:  For splitting sets, let $\F_i'$ be the set of $k2^{m'}$ cycles decomposing the tori $ \bigcup_{C\in \F_i } C \sq C$. Since the horizontal cycles in the tori have distance regular representing sets,  Proposition~\ref{dregprod} guarantees that the $k2^{m'}$ cycles in $\F_i'$ yielded by the decomposition of the tori generated by a splitting set $\F_i$ are a $k2^{m'}$-DR-splittable decomposition of $ \bigcup_{C\in \F_i } C \sq C$.  For the values $0 \le m' \le n-1$, the values of $k2^{m'}$ take on any value of $2^m$ where $x2^{n+1} - \ell + 1 \le m \le x2^{n+1} - 1+ (n+1) - \ell$. 

Next we show this decomposition is also $2^{x2^{n+1} - \ell}$-DR-splittable: 
Since all choices of $k$ we consider are even, Proposition~\ref{kwiggle} guarantees that the set of cycles decomposing each torus in $C \sq C$ is $k/2=2^{x2^{n+1}-\ell}$-splittable, where the representing sets for each cycle contain all vertices of the cycle. 
In this case define each splitting set ${\cal H}_i'$ to be a set of $k/2$ cycles given by Proposition~\ref{kwiggle} that partition the vertices of $C \sq C$ (i.e. the set of even-indexed cycles in the $k$-wiggle decomposition or the set of odd-indexed cycles). Since the distance between consecutive vertices in the representing sets is 1, we obtain a $2^{x2^{n+1}-\ell}$-DR-splittable decomposition.

Finally we show that the decomposition is $(2^m,2^{x2^{n+1}- \ell})$-DR-splittable for all values of $m$ where $x2^{n+1}- \ell \le m \le x2^{n+1}-1+(n+1)-\ell$:
Note that the splitting sets ${\cal H}_i'$ in the $2^{x2^{n+1} - \ell}$-DR-splittable decomposition contain half the cycles in a given torus, and these cycles partition the vertices of $Q_{x2^{n+1}}$, while the splitting sets $\F_i'$ in the $2^m$-DR-splittable decomposition where $m > x2^{n+1} - \ell$ contain all cycles in one or more tori.  Thus the ${\cal H}_i'$ sets partition the $\F_i'$ sets. Thus these ${\cal H}_i'$ sets can serve as the splitting subsets for all of the other decompositions, and we have a  $(2^m,2^{x2^{n+1}-\ell})$-DR-splittable decomposition for every $x2^{n+1}-\ell \le m \le x2^{n+1}-1+(n+1)-\ell$.
This completes the proof of the claim.\\

Now, we shall prove the statement of the lemma. Fix an integer $\ell$, $\ell \geq 2x$. Let $n$ be a positive integer such that $2x \le \ell \le x2^n$.  Let $n'$ be a positive integer such that $x2^{n'-1} < \ell \leq x2^{n'}$.  We see that $n\geq n'$.
We shall prove the statement of the proposition by induction on $n-n'$.
If $n-n'=0$, i.e., $n=n'$,  we are done by the claim.
Assume that the statement of the lemma holds for $n\geq n'$, i.e. $Q_{x2^n}$ has a $(2^m,2^{x2^n - \ell})$-DR-splittable decomposition into cycles of length $2^\ell$ for every $x2^n-\ell \le m \le x2^n-1+n-\ell$. We now prove the statement for $n+1$.

Case 1.  $n<\ell$. These cases are represented by the blue dots in Figure~\ref{schematic}.
Since $Q_{x2^{n+1}} = Q_{x2^n} \sq Q_{x2^n}$ and $|V(Q_{x2^n})| = 2^{x2^n}$, applying Proposition~\ref{dec} with $a= 2^{m'}$ for $x2^n-\ell \le m' \le x2^n-1+n-\ell$ and $b = 2^{x2^n - \ell}$ gives a $(a', b')$-DR-splittable decomposition where $b' = 2^{x2^n-\ell}2^{x2^n} = 2^{x2^{n+1}-\ell}$, and $a'$ can be $2^m$ for any value of $m$ from $(x2^n-\ell)+x2^n = x2^{n+1}-\ell$ to  $(x2^n-1+n-\ell)+x2^n =  x2^{n+1}-2 + (n+1) - \ell$. 
It remains to show that $Q_{x2^{n+1}}$ has a  $(2^{x2^{n+1}-1 + (n+1) - \ell},2^{x2^{n+1} - \ell})$-DR-splittable decomposition.
Applying Proposition~\ref{GboxG} to $Q_{x2^{n+1}} = Q_{x2^n} \sq Q_{x2^n}$ with $a = 2^{x2^n -1+ n - \ell}$, $b = 2^{x2^n - \ell}$, and $|V(G)| = |V(Q_{x2^n})| = 2^{x2^n}$, 
we get an $(a', b')$-DR-splittable decomposition with 
\[a'=2a |V(Q_{x2^n})| = 2\cdot 2^{x2^n -1+ n - \ell }  \cdot 2^{x2^n}= 2^{x2^{n+1} -1+ (n+1) - \ell}\]
and 
\[b'= b|V(Q_{x2^n})| = 2^{x2^{n+1} -\ell}.\]

Case 2.  $n \ge \ell$. These cases are represented by the red squares in Figure~\ref{schematic}, and follow from applying Proposition~\ref{dec} exactly as in Case 1. Since $n \ge \ell$, in this case $\min\{x2^n-1, x2^n-1 +n -\ell\}= x2^n-1$, so Proposition~\ref{GboxG} is not needed.
\end{proof}

\begin{figure}

\begin{center}

\begin{tikzpicture}[scale=.4]

\draw[thick,->] (0,0) -- (0,18) node[anchor=south east] {$n$};
\draw[thick,->] (0,0) -- (18,0) node[anchor=north west] {$\ell$};

\foreach \x in {0,2, 4, 6, 8, 10, 12, 14, 16}
   \draw (\x cm,1pt) -- (\x cm,-1pt) node[anchor=north] {$\x$};
\foreach \y in {0,2, 4, 6, 8, 10, 12, 14, 16}
    \draw (1pt,\y cm) -- (-1pt, \y cm) node[anchor=east] {$\y$};

    \foreach \x in {2}{
      \node[style={circle, draw, fill=white},scale=.4] at (\x,1){};
}

    \foreach \x in {2}{
    \foreach \y in {2}{
      \node[style={circle,fill=blue},scale=.4] at (\x,\y){};
}
}

    \foreach \x in {3,4}{
      \node[style={circle, draw, fill=white},scale=.4] at (\x,2){};
}

    \foreach \y in {3}{
    \foreach \x in {3,4}{
      \node[style={circle,fill=blue},scale=.4] at (\x,\y){};
}
}

    \foreach \x in {5, ..., 8}{
      \node[style={circle, draw, fill=white},scale=.4] at (\x,3){};
}

    \foreach \y in {4}{
    \foreach \x in {4, ..., 8}{
      \node[style={circle,fill=blue},scale=.4] at (\x,\y){};
}
}

    \foreach \x in {9, ..., 16}{
      \node[style={circle, draw, fill=white},scale=.4] at (\x,4){};
}

    \foreach \x in {5, ..., 16}{
    \foreach \y in {5, ..., \x}{
      \node[style={circle,fill=blue},scale=.4] at (\x,\y){};
}
}

    \foreach \x in {2, ..., 15}{
    \foreach \y in {\x, ..., 15}{
      \node[style={rectangle,fill=red},scale=.4] at (\x,\y+1){};
}
}

\draw[thick, dotted] (0,0) -- (16,16) {};

\node at (16, 16) [label=above right:{$n=\ell$}]{};

\draw[thick,->] (17,7) -- (17,9);
\node at (17,8) [label=right:{{\color{blue}Propositions~\ref{dec} and \ref{GboxG}}}]{};

\draw[thick,->] (3,17) -- (3,19);
\node at (3,18) [label=right:{{\color{red}Proposition~\ref{dec}}}]{};

\end{tikzpicture}
\end{center}
\caption{Schematic for proof of Lemma~\ref{m1m2GENERAL} in the case $x=1$.  The lemma in the cases $(\ell,n)$ represented by the empty dots are proved in the initial claim.  Then for a given $(\ell,n)$ where $n < \ell$ where Lemma~\ref{m1m2GENERAL} holds, Propositions~\ref{dec} and \ref{GboxG} are used in the induction to prove the lemma in the case $(\ell,n+1)$ (blue dots below the line $n=\ell$).  Finally, for a given $(\ell,n)$ where $n \ge \ell$ where Lemma~\ref{m1m2GENERAL} holds, only Proposition~\ref{dec} is needed in the induction to prove the lemma in the case $(\ell,n+1)$ (red squares above the line $n=\ell$). }\label{schematic}
\end{figure}
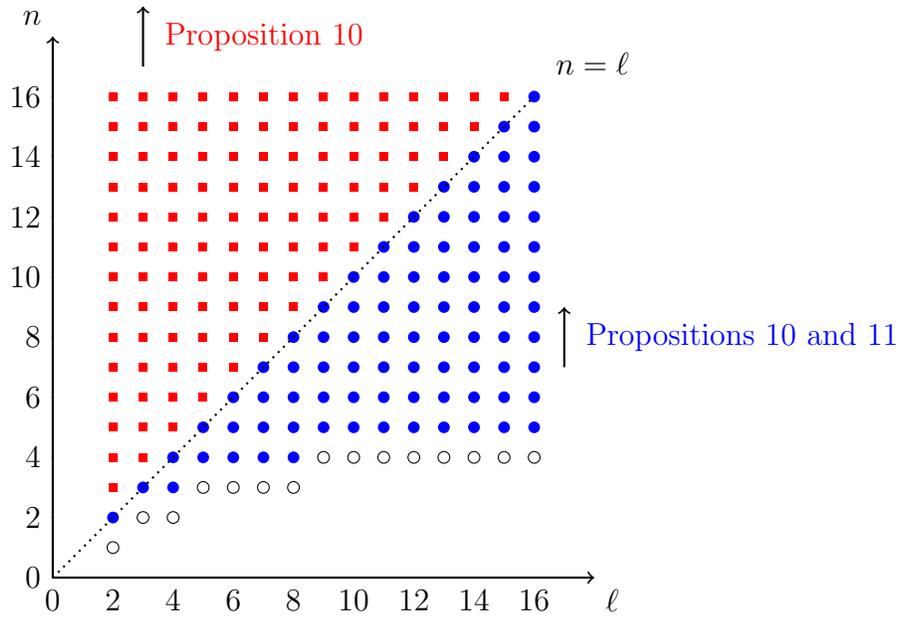



\subsection{Proofs of Theorems \ref{main}, \ref{paths}, and Corollary \ref{upshot}} \label{importantproofs}


\begin{proof}[Proof of Theorem \ref{main}]

We actually prove the following stronger statement: Let $n =xy2^\alpha$,  where $x, y \ge 1$ are odd, and $\alpha \ge 1$. Suppose $y$ has binary representation $y = 2^{i_1} + 2^{i_2} + \cdots + 2^{i_j}$, where $i_1 > i_2 > \cdots > i_j=0$. Then for $0 \le q \le 
n - i_1 - 2xj$, $Q_n$ has a $2^{j-1+q}$-splittable decomposition into $x2^{i_1 + \alpha +j-2+q}$ cycles of the same length.

We shall use induction on $j$.

Base case $j=1$.  If $j=1$, then  $y=2^0=1$, so $i_1 = 0$ and  $n=x2^{\alpha}$, where $\alpha \ge 1$.   Lemma~\ref{m1m2GENERAL} implies that $Q_{x2^{\alpha}}$ has a $2^{x2^{\alpha} - \ell}$-splittable decomposition into $x2^{x2^{\alpha} - 1 + \alpha - \ell}$ cycles  of length $2^\ell$ for each $2x \le \ell \le x2^{\alpha}$. 
Assigning $\ell$ all values in the range from $2x$ to $x2^{\alpha}$ gives all required decompositions, from a $2^{x2^{\alpha} - \ell} = 2^{x2^{\alpha} - 2x} = 2^{j-1 + (x2^{\alpha}- i_1 - 2xj)}$-splittable decomposition into $x2^{x2^{\alpha} - 1 + \alpha - \ell } = x2^{x2^{\alpha} - 1+ \alpha - 2x} =  x2^{i_1 + \alpha+j-2+ (x2^{\alpha} -i_1 - 2xj)}$ cycles when $\ell = 2x$, to a $2^{x2^{\alpha} - \ell} = 2^0 = 2^{j-1 + 0}$-splittable decomposition into 
$x2^{x2^{\alpha} - 1+ \alpha - \ell} = x2^{\alpha - 1} = x2^{i_1 + \alpha + j - 2 + 0}$ cycles when $\ell = x2^{\alpha}$.

Inductive step: Let $n = xy2^{\alpha} = x(2^{i_1} + 2^{i_2} + \cdots + 2^{i_j})2^\alpha$ with $j >1$.  Then $Q_n = Q_{x2^{i_1 + \alpha}} \sq Q_{x(2^{i_2} + \cdots + 2^{i_j})2^\alpha}$, so we seek to apply Lemma~\ref{p7gen} with $G = Q_{x2^{i_1+\alpha}}$ and $G' =  Q_{x(2^{i_2} + \cdots + 2^{i_j})2^\alpha}$.   

By Lemma~\ref{m1m2GENERAL}, $Q_{x2^{i_1 + \alpha}}$ has a $(2^{x2^{i_1 + \alpha}  - \ell + Z},2^{x2^{i_1+ \alpha}-\ell})$-DR-splittable decomposition into $x2^{x2^{i_1+ \alpha} - 1 + i_1 + \alpha - \ell}$ cycles, where $2x\le \ell \le x2^{i_1+ \alpha}$ and $0 \le Z \le \min\{\ell-1,i_1 + \alpha -1\}$.  We will choose $Z=i_1-i_2$ and thus for the remainder of the proof we will enforce the restriction that $2x + (i_1 - i_2)  \le \ell$, simultaneously ensuring that $2x \le \ell$ and $Z=i_1-i_2 \le \ell-1$.

By the inductive hypothesis, $Q_{x(2^{i_2} + \cdots + 2^{i_j})2^\alpha}$ has a $2^{j-2+q}$-splittable decomposition into $x2^{i_2 + \alpha+j-3+q}$ cycles, where $0 \le q \le x(2^{i_2} + \cdots + 2^{i_j})2^\alpha - i_2 - 2x(j - 1)$.

Let $a = 2^{x2^{i_1 + \alpha}+i_1 - i_2 -\ell}$, $b = 2^{x2^{i_1 + \alpha}-\ell}$, $m = x2^{i_2 + \alpha-1}$, and $c =2^{j-2+q}$.  Then $Q_{x2^{i_1 + \alpha}}$ has an $(a,b)$-DR-splittable decomposition into $a m$ cycles, and $Q_{x(2^{i_2} + \cdots + 2^{i_j})2^\alpha}$ has a $c$-splittable decomposition into $c m$ cycles.  Since $\ell > i_1-i_2$, $a =2^{x2^{i_1+ \alpha}+i_1 - i_2 - \ell}$ divides $|V(Q_{x2^{i_1+ \alpha}})| = 2^{x2^{i_1 + \alpha}}$ with even quotient, so the representing sets in the decomposition of $Q_{x2^{i_1 + \alpha}}$ have even cardinality at least two.  Similarly, since 
\begin{align*}
{j-2+q} &\le {j-2 + x(2^{i_2} + \cdots + 2^{i_j})2^\alpha - i_2 - 2x(j - 1)} \\
&\le {x(2^{i_2} + \cdots + 2^{i_j})2^\alpha - (2x-1)(j - 1)-1} \\
&<  {x(2^{i_2} + \cdots + 2^{i_j})2^\alpha}, 
\end{align*}
$c = 2^{j-2+q}$ divides $|V(Q_{x(2^{i_2} + \cdots + 2^{i_j})2^\alpha})| = 2^{x(2^{i_2} + \cdots + 2^{i_j})2^\alpha}$  with even quotient, so the representing sets in the decomposition of $Q_{x(2^{i_2} + \cdots + 2^{i_j})2^\alpha}$ have even cardinality at least two. Thus we can apply  Lemma~\ref{p7gen} with $G = Q_{x2^{i_1 + \alpha}}$ and $G' =  Q_{x(2^{i_2} + \cdots + 2^{i_j})2^\alpha}$ to obtain a $2bc$-splittable decomposition into $2 m a c$ cycles.  Here 
\[2bc = 2\cdot 2^{x2^{i_1 + \alpha} - \ell} \cdot 2^{j-2+q} =2^{x2^{i_1 + \alpha} - \ell + j - 1 + q}\]
 and 
\[2 m a c = 2 \cdot x2^{i_2 + \alpha -1}   \cdot 2^{x2^{i_1+ \alpha}+i_1 - i_2 - \ell }  \cdot 2^{j-2+q} =  x2^{x2^{i_1 + \alpha} + \alpha - \ell + i_1 + j - 2 + q}.\]

Letting the parameters $\ell $ and $q$ range over $2x + i_1-i_2  \le \ell \le x2^{i_1 + \alpha}$ and $0 \le q \le   x(2^{i_2} + \cdots + 2^{i_j})2^\alpha - i_2 - 2x(j - 1)$ gives

%

\[  2^{j-1 + 0}     \le 2bc \le 2^{j-1 + (n- i_1 - 2xj)} \]
and 
\[ x2^{i_1 + \alpha + j - 2 + 0} \le 2mac \le x2^{i_1 + \alpha + j - 2 +(n- i_1 - 2xj)}.\]

The lower bounds are obtained when $\ell = x2^{i_1 + \alpha}$ and $q=0$, while the upper bounds are obtained when $\ell =2x + i_1-i_2$ and $q =   x(2^{i_2} + \cdots + 2^{i_j})2^\alpha - i_2 - 2x(j - 1)$.
\end{proof}

%
%

\begin{proof}[Proof of Corollary~\ref{upshot}]
Let $n=xy2^\alpha$ be even. Let $x$ and $y$ be odd,  with $y = 2^{i_1} + \cdots + 2^{i_j}$, $i_1 > i_2> \cdots >i_j =0$. Setting $q=0$ in Theorem \ref{main} gives a decomposition of $Q_n$ into cycles of length $\ell = y2^{n-i_1-j+1} = y2^{n+1}/2^{i_1+j}$.
Since $i_1$ and $j$ are each at most $\log_2 y$,  we have $2^{i_1+j} \leq y^2$. 
Thus $\ell \geq y 2^{n+1} /y^2  = 2^{n+1}/y \geq 2^{n+1}/n$.
%
\end{proof}

\begin{proof}[Proof of Theorem~\ref{paths}]
Let $n$ be even, and suppose $\ell$ divides $|E(Q_n)|=n2^{n-1}$ and $\ell \le 2^n/n$.  Then there is some $m$ such that $\ell = y2^m$, where $y$ is an odd divisor of $n$.  By  Corollary \ref{upshot}, $Q_n$ is decomposable into cycles of length $\ell'= y2^{m'}$, where $\ell' \ge 2^{n+1}/n$.  Note that $\ell < \ell'$, and $\ell$ divides $\ell'$, so the cycles of length $\ell'$ can be divided into paths of length $\ell$, yielding a decomposition of $Q_n$ into paths of length $\ell$.
\end{proof} 

%
%

\section{A slight refinement of Theorem \ref{main}}\label{refinement}

Finally, we note that in the case $x=1$ it is possible to make a slightly stronger statement than Theorem~\ref{main}, which we prove here, along with a corollary.

\begin{prop}\label{cbgen}
Let $n$ be even, with binary representation $n = 2^{i_1} + 2^{i_2} + \cdots + 2^{i_j}$, where $i_1 > i_2 > \cdots > i_j$. Then for $0 \le q \le n - i_1 - j$, $Q_n$ has a $2^{j-1+q}$-splittable decomposition into $2^{i_1+j-2+q}$ cycles of the same length.
\end{prop}

Tables~\ref{tablemain} and \ref{tableprop} in the appendix give some examples of the cycle decompositions produced by Theorem~\ref{main} and Proposition~\ref{cbgen}. Note that even if we were just concerned with path decompositions of the hypercube, Theorem~\ref{main} gives some stronger results than Proposition~\ref{cbgen}.  For example, the cycle decompositions of $Q_{30}$ given by Proposition~\ref{cbgen} has cycles of length at most $15 \cdot 2^{24}$ (in the notation of Proposition~\ref{cbgen}, $i_1=4$ and $j=4$). Dividing these cycles in half gives paths with length $5(3\cdot 2^{23})$.  However as mentioned in the introduction, Theorem~\ref{main} gives cycles of length $5\cdot2^{m}$ for $m$ as large as 27.  Dividing these in half we get a path decomposition of $Q_{30}$ into paths of length $5\cdot2^{26}$, and $2^{26}>3\cdot2^{23}$.  Proposition~\ref{cbgen} gives more decompositions into short cycles in the case $x=1$.

\begin{proof}
We proceed by induction on $j$. 

Base case $j=1$. If $j=1$, then $n=2^{i_1}$, where $i_1 \ge 1$.   Lemma~\ref{m1m2GENERAL} implies that $Q_{2^{i_1}}$ has a $2^{2^{i_1} - \ell }$-splittable decomposition into $2^{2^{i_1} - 1+ i_1 - \ell}$ cycles when $2 \le \ell \le 2^{i_1}$. Assigning $\ell$ all values in the range from $i_1 + 1$ to $2^{i_1}$ gives all required decompositions, from a $2^{2^{i_1} - \ell} = 2^{2^{i_1} - i_1 - 1} = 2^{j-1 + (2^{i_1}-i_1 - j)}$-splittable decomposition into $2^{2^{i_1} - 1 + i_1 - \ell} = 2^{2^{i_1} - 2} =  2^{i_1+j-2+ (2^{i_1}-i_1 - j)}$ cycles when $\ell = {i_1}+1$, to a $2^{2^{i_1} - \ell} = 2^0 = 2^{j-1 + 0}$-splittable decomposition into $2^{2^{i_1} - 1+ i_1 - \ell } = 2^{i_1-1} = 2^{i_1+j-2+0}$ cycles when $\ell = 2^{i_1}$.

Inductive step: Let $n = 2^{i_1} + 2^{i_2} + \cdots + 2^{i_j}$, with $j >1$.  Then $Q_n = Q_{2^{i_1}} \sq Q_{2^{i_2} + \cdots + 2^{i_j}}$, so we seek to apply Lemma~\ref{p7gen} with $G = Q_{2^{i_1}}$ and $G' =  Q_{2^{i_2} + \cdots + 2^{i_j}}$.

By Lemma~\ref{m1m2GENERAL} , $Q_{2^{i_1}}$ has a $(2^{2^{i_1} -\ell + Z},2^{2^{i_1}-\ell})$-DR-splittable  decomposition into $2^{2^{i_1} - 1+ i_1 - \ell }$ cycles, where $2\le \ell \le 2^{i_1}$ and  $0 \le Z \le \min\{\ell - 1,i_1 - 1\}$.  We will choose $Z=i_1-i_2$ and thus for the remainder of the proof we have the restriction $i_1 - i_2 +1 \le \ell$, ensuring $2 \le \ell$ and $Z \le \ell -1$.

By the inductive hypothesis, $Q_{2^{i_2} + \cdots + 2^{i_j}}$ has a $2^{j-2+q}$-splittable decomposition into $2^{i_2+j-3+q}$ cycles, where $0 \le q \le
2^{i_2} + \cdots + 2^{i_j} - i_2 - j + 1$.

Let $a = 2^{2^{i_1}+i_1 - i_2 - \ell}$, $b = 2^{2^{i_1} - \ell}$,  $m = 2^{i_2-1}$, and $c =2^{j-2+q}$.  Then $Q_{2^{i_1}}$ has an $(a,b)$-DR-splittable decomposition into $a m$ cycles, and $Q_{2^{i_2} + \cdots + 2^{i_j}}$ has a $c$-splittable decomposition into $c m$ cycles.  Since $\ell > i_1-i_2$, $a =2^{2^{i_1}+i_1 - i_2 -\ell}$ divides $|V(Q_{2^{i_1}})| = 2^{2^{i_1}}$ with even quotient, so the representing sets in the decomposition of $Q_{2^{i_1}}$ have even cardinality at least two.  Similarly, since $q < 2^{i_2} + \cdots + 2^{i_j} - j + 2$, $c=2^{j-2+q}$ divides $|V(Q_{2^{i_2} + \cdots + 2^{i_j}})| = 2^{2^{i_2} + \cdots + 2^{i_j}}$  with even quotient, so the representing sets in the decomposition of $Q_{2^{i_2} + \cdots + 2^{i_j}}$ have even cardinality at least two. Thus we can apply  Lemma~\ref{p7gen} with $G = Q_{2^{i_1}}$ and $G' =  Q_{2^{i_2} + \cdots + 2^{i_j}}$ to obtain a $2bc$-splittable decomposition into $2mac$ cycles.  Here 
\[2bc = 2\cdot 2^{2^{i_1}-\ell} \cdot 2^{j-2+q} =2^{2^{i_1} - \ell + j - 1 + q}\]
 and 
\[2mac = 2 \cdot 2^{i_2-1}   \cdot 2^{2^{i_1}+i_1 - i_2 -\ell}  \cdot 2^{j-2+q} =  2^{2^{i_1} - \ell + i_1 + j - 2 + q}.\]

Letting the parameters $\ell$ and $q$ range over $i_1-i_2 + 1 \le \ell \le 2^{i_1}$ and $0 \le q \le
2^{i_2} + \cdots + 2^{i_j} - i_2 - j + 1$ gives 

%

\[  2^{j-1 + 0}     \le 2bc \le 2^{j-1 + (n- i_1 -j)} \]
and 
\[ 2^{i_1 + j - 2 + 0} \le 2mac \le 2^{i_1 + j - 2 +(n- i_1 - j)}.\]

The lower bounds are obtained when $\ell=2^{i_1}$ and $q=0$, while the upper bounds are obtained when $\ell=i_1-i_2 + 1$ and $q =  2^{i_2} + \cdots + 2^{i_j} - i_2 - j + 1$.
\end{proof}

The following corollary shows that we get a decomposition of $Q_n$ into almost all cycles whose length divides $n2^{n-1}$ and is divisible by $2n$.

\begin{corollary}\label{multn}
Let $n$ be even and $\ell = n2^m$.  Then there is a decomposition of $Q_n$ into cycles of length $\ell$ if $2n \le \ell \le 2^n/n$.
\end{corollary}

\begin{proof}
Let $n = 2^{i_1} +\cdots + 2^{i_j}$ be even, with $i_1 > \cdots > i_j$. By Proposition~\ref{cbgen}, $Q_n$ can be decomposed into  $2^{i_1+j-2+q}$ cycles of the same length, where $0 \le q \le n - i_1 - j$.  Since $Q_n$ has $n2^{n-1}$ edges, this gives cycles of length $n2^{n-1}/2^{i_1+j-2+q} = n2^{n-i_1-j+1-q}$ for $0 \le q \le n - i_1 - j$.  Letting $q$ vary from $0$ to $n - i_1 - j$ gives cycles of length $n2^m$ for all $m$ from 1 (when $q=n-i_1-j$) to ${n-i_1-j+1}$ (when $q=0$). As in the proof of Corollary~\ref{upshot}, $n2^{n-i_1-j+1} \ge 2^{n+1}/n$.
\end{proof}

%
%


\section{Acknowledgements}
The second author was supported by a DAAD Award: Research Stays for University Academics and Scientists (Program 57381327) to visit Eberhard Karls University of T{\"u}bingen and Karlsruhe Institute of Technology.  The work of the third author was supported by the grant IBS-R029-C1.  We thank the anonymous referees for a careful reading and constructive suggestions improving the presentation.



\appendix
\section{Numerical examples}

\begin{table}[H]
\begin{center}
\begin{tabular}{|l |l |l |l |l |l |l |l |l |}
$n$ & $\alpha$ & $x$ & $y$ & $i_1$ & $j$ & $n-i_1 - 2xj$  & Number of cycles & Cycle lengths \\ \hline
14 & 1 & 1 & 7 & 2 & 3 & 6 & $\{2^{q} ~:~ 4 \le q \le 10\}$ &  $\{7 \cdot 2^{m} ~:~ 4 \le m \le 10\}$ \\
14 & 1 & 7 & 1 & 0 & 1 & 0 & $\{7 \cdot 2^{q} ~:~ 0 \le q \le 0\}$ &  $\{2^{14}\}$ \\ \hline
30 & 1 & 1 & 15 & 3 & 4 & 19 & $\{2^{q} ~:~ 6 \le q \le 25\}$ &  $\{15 \cdot 2^{m} ~:~ 5 \le m \le 24\}$ \\
30 & 1 & 3 & 5 & 2 & 2 & 16 & $\{3 \cdot 2^{q} ~:~ 3 \le q \le 19\}$ &  $\{5 \cdot 2^{m} ~:~ 11 \le m \le 27\}$ \\
30 & 1 & 5 & 3 & 1 & 2 & 9 & $\{5 \cdot 2^{q} ~:~ 2 \le q \le 11\}$ &  $\{3 \cdot 2^{m} ~:~ 19 \le m \le 28\}$ \\
30 & 1 & 15 & 1 & 0 & 1 & 0 & $\{15 \cdot 2^{q} ~:~ 0 \le q \le 0\}$ &  $\{2^{30}\}$ \\ \hline
180 & 2 & 1 & 45 & 5 & 4 & 167 & $\{2^{q} ~:~ 9 \le q \le 176\}$ &  $\{45 \cdot 2^{m} ~:~ 5 \le m \le 172\}$ \\
180 & 2 & 3 & 15 & 3 & 4 & 153 & $\{3 \cdot 2^{q} ~:~ 7 \le q \le 160\}$ &  $\{15 \cdot 2^{m} ~:~ 21 \le m \le 174\}$ \\
180 & 2 & 9 & 5 & 2 & 2 & 142 & $\{9 \cdot 2^{q} ~:~ 4 \le q \le 146\}$ &  $\{5 \cdot 2^{m} ~:~ 35 \le m \le 177\}$ \\
180 & 2 & 5 & 9 & 3 & 2 & 157 & $\{5 \cdot 2^{q} ~:~ 5 \le q \le 162\}$ &  $\{9 \cdot 2^{m} ~:~ 19 \le m \le 176\}$ \\
180 & 2 & 15 & 3 & 1 & 2 & 119 & $\{15 \cdot 2^{q} ~:~ 3 \le q \le 122\}$ &  $\{3 \cdot 2^{m} ~:~ 59 \le m \le 178\}$ \\
180 & 2 & 45 & 1 & 0 & 1 & 90 & $\{45 \cdot 2^{q} ~:~  1 \le q \le 91\}$ &  $\{2^{m} ~:~ 90 \le m \le 180\}$ \\
\end{tabular}
\end{center}
\caption{The cycle lengths of the cycle decompositions of $Q_n$ in Theorem~\ref{main}.}\label{tablemain}
\end{table}

\begin{table}[H]
\begin{center}
\begin{tabular}{|l |l |l |l |l |l |}
$n$ & $i_1$ & $j$ & $n-i_1 - j$  & Number of cycles & Cycle lengths \\ \hline
14 & 3 & 3 & 8 & $\{2^{q} ~:~ 4 \le q \le 12\}$ &  $\{7 \cdot 2^{m} ~:~ 2 \le m \le 10\}$ \\ \hline
30 & 4 & 4 & 22 & $\{2^{q} ~:~ 6 \le q \le 28\}$ &  $\{15 \cdot 2^{m} ~:~ 2 \le m \le 24\}$ \\ \hline
180 & 7 & 4 & 169 & $\{2^{q} ~:~ 9 \le q \le 178\}$ &  $\{45 \cdot 2^{m} ~:~ 3 \le m \le 172\}$ \\
\end{tabular}
\end{center}
\caption{The cycle lengths of the cycle decompositions of $Q_n$ in Proposition~\ref{cbgen}.}\label{tableprop}
\end{table}

\end{document}